\documentclass[11pt,a4paper,leqno]{amsart}

\usepackage[latin1]{inputenc}
\usepackage[T1]{fontenc}
\usepackage{amsfonts}
\usepackage{amsmath}
\usepackage{amsthm}
\usepackage{amssymb}
\usepackage{eurosym}
\usepackage{mathrsfs}
\usepackage{palatino}
\usepackage{color}
\usepackage{esint}
\usepackage{url}
\usepackage{verbatim}
\usepackage{enumerate}
\usepackage[dvipsnames]{xcolor}

\usepackage{caption,float}
\usepackage{placeins}
\usepackage{afterpage}

\usepackage{subfig}

\usepackage[pagebackref,hypertexnames=false, colorlinks, citecolor=black, linkcolor=black, urlcolor=black]{hyperref}
\usepackage[abbrev]{amsrefs}
\newfloat{fig}{tb}{lop}[section]
\DeclareCaptionLabelFormat{Ucase}{Figure~#2}
\usepackage{graphicx}
\graphicspath{ {./pics_opp/} }

\newcommand{\R}{\mathbb{R}}
\newcommand{\C}{\mathbb{C}}

\newcommand{\N}{\mathbb{N}}

\newcommand{\scrS}{\mathscr{S}}
\newcommand{\calF}{\mathcal{F}}

\newcommand{\calK}{\mathcal{K}}

\newcommand{\calR}{\mathcal{R}}

\newcommand{\calO}{\mathcal{O}}

\newcommand{\diam}{\operatorname{diam}}

\newcommand{\vare}{\varepsilon}

\numberwithin{equation}{section}

\newcommand{\ud}[0]{\,\mathrm{d}}

\newcommand{\dist}[0]{\operatorname{dist}}

\newcommand{\op}[1]{\operatorname{#1}}

\newcommand{\abs}[1]{|#1|}
\newcommand{\babs}[1]{\big|#1\big|}
\newcommand{\Babs}[1]{\Big|#1\Big|}
\newcommand{\norm}[2]{|#1|_{#2}}

\newcommand{\Norm}[2]{\|#1\|_{#2}}
\newcommand{\bNorm}[2]{\big\|#1\big\|_{#2}}
\newcommand{\BNorm}[2]{\Big\|#1\Big\|_{#2}}

\newcommand{\ave}[1]{\langle #1\rangle}
\newcommand{\bave}[1]{\big\langle #1\big\rangle}
\newcommand{\Bave}[1]{\Big\langle #1\Big\rangle}

\newcommand{\BMO}[0]{\operatorname{BMO}}

\newcommand{\VMO}[0]{\operatorname{VMO}}

\newcommand{\CMO}[0]{\operatorname{CMO}}

\newcommand{\supp}[0]{\operatorname{spt}}
\newcommand{\loc}[0]{\operatorname{loc}}

\newcommand{\sign}[0]{\operatorname{sgn}}

\newcommand{\ch}[0]{\operatorname{ch}}

\newcommand{\calD}[0]{\mathcal{D}}

\newcommand{\wt}[1]{{\widetilde{#1}}}

\swapnumbers
\theoremstyle{plain}
\newtheorem{thm}[equation]{Theorem}
\newtheorem{lem}[equation]{Lemma}
\newtheorem{prop}[equation]{Proposition}

\newtheorem{que}[equation]{Question}

\newtheoremstyle{named}{}{}{\itshape}{}{\bfseries}{.}{.5em}{\thmnote{#3}#1}
\theoremstyle{named}

\theoremstyle{definition}
\newtheorem{defn}[equation]{Definition}

\theoremstyle{remark}

\newtheorem{rem}[equation]{Remark}

\title[On commutators along curves]{On the $L^p$-to-$L^q$ boundedness and compactness of commutators along monomial curves}

\author{Tuomas Oikari}

\address[T.O.]{Department of Mathematics and Statistics, University of Jyv\"{a}skyl\"{a}, Seminaarinkatu 15, 40014 Jyv\"{a}skyl\"{a}n yliopisto, Finland}
\email{tuomas.v.oikari@jyu.fi}

\makeatletter
\@namedef{subjclassname@2020}{%
	\textup{2020} Mathematics Subject Classification}
\makeatother


\begin{document}
	
\begin{abstract} 
		For exponents $p,q\in (1,\infty),$ we study the $L^p$-to-$L^q$ boundedness and compactness of the commutator 
		$
		[b,H_{\gamma}] = bH_{\gamma} - H_{\gamma}b,
		$ where $H_{\gamma}$ is the Hilbert transform along the monomial curve $\gamma$ and the function $b$ is complex valued.

	We obtain a sparse form domination of the commutator and show that $b\in \BMO^{\gamma,\alpha}$ is a sufficient condition for boundedness in a half-open range of exponents $q\in[p,p+p(\Xi));$ in doing so we obtain a genuinely fractional case $q>p$ and also give a new proof of the case $q=p,$ due to Bongers, Guo, J. Li and Wick \cite{BGLW2021}. 
	Involving the $L^p$ improving phenomena of single scale averages, we show that $b\in \VMO^{\gamma,\alpha}$ is a sufficient condition for compactness in the same range of exponents $q\in[p,p+p(\Xi))$ as in the case of boundedness.
	
	For $q<p,$ we show that $b\in \dot L^r$ is a sufficient condition for compactness; the argument is new and we use it to obtain a new short proof of the same implication for commutators of Calderón-Zygmund operators, recently due to Hyt\"{o}nen, K. Li, Tao and Yang \cite{HLTY2022}.
	
	For necessity, and now restricting to the plane, we develop the approximate weak factorization argument to the extent that it now works for all monomial curves, thus removing the extra assumption present in our earlier work \cite{Oik2022}
	that the graph $\gamma(\R)\subset\R^2$ intersects adjacent quadrants. 
	As a consequence in the plane, we obtain that all of the above sufficiency conditions are necessary.
	
	We present several open problems.
\end{abstract}

\maketitle

\setcounter{secnumdepth}{4}
\setcounter{tocdepth}{2}
\tableofcontents

\section{Introduction}\label{subsect:14}
\subsection{Main results}\label{sect:intro}
We study the commutator
\begin{align*}
	[b,H_{\gamma}]f(x) = b(x)H_{\gamma}f(x)-H_{\gamma}(bf)(x)
\end{align*} 
of a complex valued symbol $b:\R^n\to\C$ against the Hilbert transform along a monomial curve $\gamma,$
\begin{align*}
	H_{\gamma}f(x) = p.v.\int_{\R}f(x-\gamma(t))\frac{\ud t}{t},
\end{align*}
where $\gamma$ is specified by three parameters $\beta\in\R_+^n$ and $\delta,\vare\in \{-1,1\}^n$ as 
\begin{align*}
\gamma:\R\to\R^n,\qquad \gamma(t) = 
	\begin{cases}
		(\varepsilon_1\abs{t}^{\beta_1},\dots,\varepsilon_n\abs{t}^{\beta_n}),\quad t > 0, \\ 
		(\delta_1\abs{t}^{\beta_1},\dots,\delta_n\abs{t}^{\beta_n}),\quad t\leq 0,
	\end{cases}
\end{align*}
under the constraints that $0<\beta_1<\cdots<\beta_n$
and there exists an index $i$ so that $\varepsilon_i\not=\delta_i.$ 

For the rest of this article we assume that $p,q\in (1,\infty),$ that $\gamma$ stands for a monomial curve, and that the symbol $b$ of the commutator is complex valued and locally integrable; we make these background assumptions on everything we write.
The main contributions of this article are the following three Theorems \ref{thm:main:bdd}, \ref{thm:main:comp} and \ref{thm:main:q<p} that we now state partially, in practice we prove more.
\begin{thm}\label{thm:main:bdd} For each $p\in(1,\infty),$ there exists $p(\Xi)>p$ such that for all $q\in [p,p(\Xi))$ we have
		\begin{align}
		\Norm{[b,H_{\gamma}]}{L^p(\R^2)\to L^q(\R^2)}\sim 
			\Norm{b}{\BMO^{\gamma,\alpha}(\R^2)}.
	\end{align}
\end{thm}
\begin{thm}\label{thm:main:comp} Under the same restrictions as in Theorem \ref{thm:main:bdd}, we have 
	\begin{align}
		[b,H_{\gamma}] \in \calK(L^p(\R^2), L^q(\R^2)) \Longleftrightarrow
		b\in 
			\VMO^{\gamma,\alpha}(\R^2),
	\end{align}
	where
	 $\calK$ stands for the class of compact operators between the given spaces.
\end{thm}
\begin{thm}\label{thm:main:q<p} Let $1<q<p<\infty$ and $1/q=1/r+1/p.$ 
Then, 
\begin{align}
	\Norm{[b,H_{\gamma}]}{L^p(\R^2)\to L^q(\R^2)} < \infty  \Longleftrightarrow b\in\dot L^r(\R^2)\Longrightarrow  [b,H_{\gamma}]\in \calK(L^p(\R^2),L^q(\R^2)).
\end{align}
\end{thm}
When speaking of sufficiency claims, we mean that the testing conditions $\BMO^{\gamma,\alpha},$ $\VMO^{\gamma,\alpha}$ and $\dot L^r$ on the symbol (of the commutator) are sufficient to conclude boundedness or compactness (of the commutator); necessity goes the other way around, from boundedness or compactness, we conclude these necessary testing conditions on the symbol.  
Concerning Theorems \ref{thm:main:bdd} and \ref{thm:main:comp}, we prove all of the sufficiency claims in arbitrary dimensions $n\geq 2$, and only for this direction we need to assume $q<p(\Xi).$ 
Concerning Theorem \ref{thm:main:q<p}, we obtain that $b\in \dot L^r(\R^n)$ is sufficient for compactness for all $n\geq 2.$ 
In the other direction, and now concerning all three Theorems \ref{thm:main:bdd}, \ref{thm:main:comp} and \ref{thm:main:q<p}, we obtain the necessity claims in the two-weight Bloom setting, although then we must currently restrict to the plane $n=2.$

There are some natural research questions that we have not settled in this article and that we think merit to be carefully studied on their own. As we go, we introduce these in Section \ref{sect:mainresults} below.

\subsection{Notation} Before going into details we recall basic facts about $\R^n$ with anisotropic dilations and fix notational conventions.
\begin{itemize}
	\item With $\lambda>0,$ $\beta\in\R_+^n$ and $x\in\R^n,$ we denote the anisotropic dilation with
	$$
	\lambda^{\beta}x = (\lambda^{\beta_1}x_1,\dots,\lambda^{\beta_n}x_n).
	$$ 
	Let $r>0$ be the unique solution of the equation $\abs{r^{-\beta}x} = 1$ and define
	\[
	\abs{x}_{\beta} :=r.
	\]
	Alternatively, we can use any equivalent quasi-metric,
	\begin{align}\label{eq:quasinorms}
		\abs{x}_{\beta}\sim \sum_{j=1}^n\abs{x_j}^{1/\beta_j}\sim \max_{j=1,\dots, n}\abs{x_j}^{1/\beta_j}.
	\end{align}
	\item Given a set $E\subset\R^n$ and a specified centre point $c_E\in\R^n,$ we define
	$$
	\lambda^{\beta}E = \lambda^{\beta}_{c_E}E = \big\{ c_E + \lambda^{\beta}(x-c_E): x\in E\big\},
	$$
	 the centric dilation.
It will always be clear from the context what the centre of dilation is, and when we dilate a set with respect to the origin we notate $\lambda^{\beta}_{0}E.$ 
	\item We denote the $\beta$-ball with centre $c$ and radius $r$ with  
	\[
	B_{\beta}(c,r) := r^{\beta}B(c,1) = \{ x\in\R^n: \abs{x-c}_{\beta}<r\}.
	\]
	We let $\calR^{\beta}$ denote the collection of all $\beta$-cubes,  i.e. rectangles
	$
	Q=I_1\times\dots\times I_n
	$
	 parallel to the coordinate axes such that
	\begin{align*}
		\ell(Q) := \ell(I_1)^{1/\beta_1}=\dots=\ell(I_n)^{1/\beta_n}.
	\end{align*}
	Notice that $\abs{Q} = \ell(Q)^{\abs{\beta}},$ where  $\abs{\beta} := \sum_{j=1}^n\beta_j.$
	
	We denote with $Q_{\beta}(x,r)$ the unique $\beta$-cube with centre $x$ and sidelength $r.$ Then, cubes and balls are comparable
	\[
	B_{\beta}(x,r)\subset Q_{\beta}(x,r)\subset B_{\beta}(x,Cr),\qquad \abs{B_{\beta}(c,r)}\sim	\abs{Q_{\beta}(x,r)} =  r^{\abs{\beta}},
	\]
	for some absolute constant $C>0$. 
	\item 	
	If $\beta$ is associated to a monomial curve $\gamma,$ we call $\beta$-objects $\gamma$-objects. e.g. $\calR^{\beta}= \calR^{\gamma}.$
	\item Whenever $\beta = (1,\dots,1)$ or $w = 1$ or $\alpha = 0,$ we freely drop these symbols from the notation, e.g. $\VMO^{(1,\dots,1),0}_1 = \VMO.$ 
	\item We denote averages as
	$$
	\langle f \rangle_A = \fint_A f= \frac{1}{|A|} \int_A f,\qquad \ave{f}_{A,t} = \ave{\abs{f}^t}_A^{1/t},
	$$
	where $|A|$ denotes the Lebesgue measure of the set $A$.
	\item We denote $A \lesssim B$, if $A \leq C B$ for some constant $C>0$ depending only on the dimension of the underlying space, on integration exponents, on sparse constants, on auxiliary functions, and on other absolute factors that we do not care about.
	Then  $A \sim B$, if $A \lesssim B$ and $B \lesssim  A.$ Moreover, subscripts on constants ($C_{a,b,c,...}$) and quantifiers ($\lesssim_{a,b,c,...}$) signify their dependence on those subscripts. 	
\end{itemize} 

\subsection{Main results in detail}\label{sect:mainresults} 
We prove some of our results in weighted settings and thus into the definitions we introduce weights, i.e. positive locally integrable functions $w:\R^n\to\R_+.$ 
	Let $\alpha\in\R$ and define 
	$$\BMO_{w}^{\gamma, \alpha}= \BMO_{w}^{\beta, \alpha}$$ through the semi-norm
	\begin{align}
		\Norm{b}{\BMO_{w}^{\beta, \alpha}} = \sup_{Q\in\calR^{\beta}}\calO_{w}^{\beta,\alpha}(b;Q),\quad \calO_{w}^{\beta,\alpha}(b;Q) = w(Q)^{-\alpha/\abs{\beta}}\Big(\frac{1}{w(Q)}\int_Q\abs{b-\ave{b}_Q}\Big),
	\end{align}
where $w(Q)=\int_Qw.$
Bongers, Guo, Li and Wick \cite{BGLW2021} obtained the upper bound
\begin{align}\label{eq:BGLW}
	\Norm{[b,H_{\gamma}]}{L^p(\R^n)\to L^p(\R^n)} \lesssim \Norm{b}{\BMO^{\gamma}(\R^n)},\qquad \BMO^{\gamma}= \BMO_{1}^{\gamma,0}
\end{align}
through the Cauchy integral argument by blackboxing the sparse form domination of $H_{\gamma}$ and the resulting weighted bounds from Cladek, Ou \cite{ClaOu2017}.
We give an alternative proof of \eqref{eq:BGLW} by providing a sparse form domination of the commutator, which also allows us to include genuinely fractional cases. The set of admissible exponents $q\geq p$ is
$$
(1/p,1/q)\in  \Xi(n)  \subset (0,1)^2,
$$
where
\begin{align*}
	\Xi(n) &=\op{int}\Omega(n)\cup \big\{\big(t,t\big): t\in (0,1)\big\}, \\
	\Omega(n) &= \overline{\operatorname{conv}}\left\{ (0,0), (1,1), \Big( \frac{n^2-n+2}{n^2+n},\frac{n-1}{n+1}\Big),  \Big( \frac{2}{n+1} \frac{2n-2}{n^2+n}\Big)  \right\},
\end{align*}
where $\overline{\operatorname{conv}}$ is the closed convex hull of the specified points, see Figure \ref{fig:ExpXi}.
\begin{thm}\label{thm:UPfrac}  Let $\gamma:\R\to\R^n$ be a monomial curve, let $(1/p,1/q)\in \Xi(n)$ and $\alpha/\abs{\beta} = 1/p-1/q.$ Then,
\begin{align}\label{eq:bound:up}
	\Norm{[b,H_{\gamma}]}{L^p(\R^n)\to L^q(\R^n)} \lesssim \Norm{b}{\BMO^{\gamma,\alpha}(\R^n)}.
\end{align}
\end{thm}
To connect Theorems \ref{thm:main:bdd} and \ref{thm:UPfrac} we define
\begin{align*}
	p(\Xi(n)) = \sup \{q\geq p: (1/p,1/q)\in\Xi(n)\}.
\end{align*}
The proof of Theorem \ref{thm:UPfrac} is based on the sparse form domination of the commutator.
\begin{thm}\label{thm:SDOM}  Let $\gamma:\R\to\R^n$ be a monomial curve and
	$\big(1/r, 1/s'\big)\in \Omega(n).$ 
	Then, for all $f,g\in L^{\infty}_c(\R^n),$ there exists a sparse collection $\mathscr{S}^{\gamma} \subset \calR^{\gamma}$ such that 
	\[
	\babs{\bave{[b,H_{\gamma}]f,g}}\lesssim 	A_{b}^{r,s}[\mathscr{S}^{\gamma}](f,g)  + 	A_{b}^{r,s*}[\mathscr{S}^{\gamma}](f,g),
	\]
	where
	\begin{equation}\label{eq:SFORMS}
		\begin{split}
			A_{b}^{r,s}[\mathscr{S}^{\gamma}](f,g) &= \sum_{Q\in \mathscr{S}^{\gamma}} \bave{(b-\ave{b}_Q)f}_{Q,r}\bave{g}_{Q,s}\abs{Q}, \\
			A_{b}^{r,s*}[\mathscr{S}^{\gamma}](f,g) &= \sum_{Q\in \mathscr{S}^{\gamma}}\bave{f}_{Q,r}\bave{(b-\ave{b}_Q)g}_{Q,s}\abs{Q}.
		\end{split}
	\end{equation}
\end{thm}

For the converse of Theorem \ref{thm:UPfrac}, we already know the matching lower bound for those monomial curves that intersect adjacent quadrants of the plane.
\begin{thm}[\cite{Oik2022}]\label{thm:adjacent} Let $\gamma:\R\to\R^2$ be a monomial curve with parts in adjacent quadrants., i.e. $\vare_i=\delta_i$ for exactly one index $i\in\{1,2\}.$  Let $p\leq q$ and $\alpha/\abs{\beta} = 1/p-1/q.$ Then,
	\begin{align*}
		\Norm{b}{\BMO^{\gamma,\alpha}(\R^2)} \lesssim \Norm{[b,H_{\gamma}]}{L^p(\R^2)\to L^q(\R^2)}.
	\end{align*}
\end{thm}
Already from Theorems \ref{thm:UPfrac} and \ref{thm:adjacent} we see a discrepancy in the obtained upper and lower bounds.
The specific form of $\Xi(n)$ stems from the fact that for these exponents we have access to a sparse form domination, which in turn is determined by the availability of smoothing estimates for single scale averaging operators. We do not know if $\Xi(n)$ is maximal with respect to the sufficiency results, but we know from \cite[Section 4.2.]{ClaOu2017} that $\Omega(n)$ is sharp with respect to obtaining a sparse form domination of $H_{\gamma}.$ By the heuristic principle that commutators of singular integrals are more singular than the singular integrals themselves, we cannot expect a better sparse form domination for the commutator than for the singular integral. However, this does not exclude the possibility that $\Xi(n),$ respectively $p(\Xi(n)),$ might not be maximal for the upper bound \eqref{eq:bound:up}.

\begin{que}\label{que:1} Let $p\in (1,\infty),$ $n\geq 2$ and $\beta\in\R_+^n$ be fixed (i.e. $\gamma$ is fixed).
Determine the maximal interval
	\begin{align}\label{eq:que:1}
		 [p,p_{\op{max}}(n,\beta)),\qquad   p_{\op{max}}(n,\beta) := \sup\{q \geq p: \eqref{eq:bound:up}\, \mbox{holds} \}.
	\end{align}
Does there hold that $p_{\op{max}}(n,\beta)=\infty?$
\end{que}

Now we turn to necessity.
One of the two main approaches to commutator lower bounds is through the approximate weak factorization (awf) argument, attributable to Uchiyama \cite{Uch1981} and Hyt\"{o}nen \cite{HyLpLq}; the other approach is known as the median method, see e.g. Lerner, Ombrosi and Rivéra-Rios \cite{LOR1}.
In detail, the main result in \cite{Oik2022} was the awf argument for those monomial curves that intersect adjacent quadrants of the plane.
In this article, we obtain the awf argument for all monomial curves in the plane, thus including the case where $\gamma(\R)\subset \R^2$ intersects opposite quadrants. Consider the parabolic curves
\begin{align}\label{eq:adjopp}
	\gamma_{\op{adj}}(t) = (t,t^2),\qquad \gamma_{\op{opp}}(t) = (t,\sign(t)t^2)
\end{align}
representative of the two cases.
 A nicety is that the proof works simultaneously for both $\gamma_{\op{adj}}$ and $\gamma_{\op{opp}}$ without resulting in cases. We now take some time to discuss a corollary of the proof, a local off-support norm for $\BMO^{\gamma}(\R^2),$ Theorem \ref{thm:osc} below. 
 Its importance lies with the fact that through it all necessity claims we prove for both boundedness and compactness follow by arguments that do not anymore have to directly see the structure of the operator; moreover, its locality allows rescaling with powers of doubling weights and thus it is as easy to obtain necessity claims in many weighted settings as in the unweighted setting.

For a bounded connected set $E\subset \R^2,$ we now specify the centre point of $E$ to be $c_{E} = (c_{\pi_1(E)},c_{\pi_2(E)}),$ where $\pi_i(x) = x_i$ is the $i$th projection and $c_{\pi_1(E)}$ is the centre point of the interval $\pi_1(E).$ 
\begin{thm}\label{thm:osc}  Let $\gamma:\R\to\R^2$ be a monomial curve with a parameter tuple $\beta.$ Let $Q\in\calR^{\gamma}.$ Then, there exist bounded connected sets
	$W_1,W_2,P \subset \R^2$ and absolute constants $\Lambda > \lambda > 1$ so that for distinct $S_1,S_2\in\{Q,W_1,W_2,P\}$ there holds that
	\begin{align}\label{eq:osc1}
		\abs{S_1}\sim \abs{S_2},\qquad \lambda^{\beta}_{c_{S_1}} S_1 \cap S_2=\emptyset,\qquad S_2\subset \Lambda^{\beta} S_1,
	\end{align}
	and there exist functions $\psi_{S_1},\psi_{S_2}$ so that 
	\begin{align}\label{eq:osc2}
		\abs{\psi_{S_i}}\lesssim 1_{S_i},\qquad i=1,2,
		\end{align}
	and
	\begin{align}\label{eq:osc3}
		\fint_Q\babs{b-\ave{b}_Q}\lesssim \frac{1}{\abs{Q}}\babs{\bave{[b,H_{\gamma}]\psi_{S_1},\psi_{S_2}}} \lesssim 	\fint_{\Lambda^{\beta} Q}\babs{b-\ave{b}_{\Lambda^{\beta}Q}}.
	\end{align}
Moreover, all the constants above are independent of $Q\in \calR^{\gamma}.$
\end{thm}
The upper bound on the line \eqref{eq:osc3} is straightforward and not needed for the main results of this article, but it is interesting that the lower bound can be reversed up to a constant dilation. The argument we give for Theorem \ref{thm:osc} is similar to that in \cite{Oik2022} and is based on finding foliations of $\gamma$-cubes by families of cut segments of $\gamma$. Without verifying the details, it seems that a similar argument should work in higher dimensions $n\geq 3,$ however as the present argument is already quite involved in the plane $n=2$, 
it would be interesting to find an alternative proof
that would not to such a great extent depend on the exact foliations present in the current argument.

\begin{que}\label{que:2}  Obtain a version of Theorem \ref{thm:osc} for $n\geq 3$ and by arguments similar to those in this article extend all the necessity results, discussed below, to dimensions $n\geq 3.$
\end{que}

We use the anisotropic and rescaled Muckenhoupt weight class and treat weights as multipliers,
\begin{align}\label{eq:Ap}
	[w]_{A_{p,p}^{\beta}} = \sup_{Q\in\calR^{\beta}} \ave{w^p}_Q^{1/p}\ave{w^{-p'}}^{1/p'},\qquad \Norm{f}{L^s_{w}} = \left(\int\abs{fw}^s\right)^{1/s}.
\end{align}
The relationship to the usual definition is as follows,
\begin{align}\label{eq:App}
	[w]_{A_{p,p}^{\beta}} = [w^p]_{A_{p}^{\beta}}^{1/p},\qquad [\sigma]_{A_{p}^{\beta}} := \sup_{Q\in\calR^{\beta}} \ave{\mu}_Q\ave{\sigma^{-\frac{p'}{p}}}^{\frac{p}{p'}}.
\end{align}
 Bloom \cite{Bloom1985} characterized the two-weight $L^p_{\mu}$-to-$L^q_{\lambda}$ boundedness of the commutator of the Hilbert transform under the assumptions $\mu,\lambda\in A_{p,p}^1.$  Questions of boundedness and compactness of commutators from $L^p_{\mu}$-to-$L^q_{\lambda}$ under the background assumptions $\mu\in A_{p,p}^{(1,\dots,1)}$ and $\lambda\in A_{q,q}^{(1,\dots,1)}$  are dubbed to be of Bloom-type,  irrespective the three ways of ordering $p,q.$
When $p=q,$ the initial work of Holmes, Lacey and Wick \cite{HLW2017} extended Bloom's result to virtually all Calderón-Zygmund operators and to the case $n\geq 2.$ Then, interest into estimates of Bloom-type renewed, see e.g. the subsequent work \cites{HLW2016,HW2018, HPW2018, LOR1, LOR2, LMV2019biparBloom, LMV2019prodBloom} of several authors.
  
	Given two weights $\mu,\lambda$ and exponents $p,q,$ we introduced \cite[Definition 2.10]{HOS2022} the following fractional Bloom weight,
	\begin{align}\label{eq:bloom:weight}
		\nu =  \big(\mu/\lambda\big)^{(1+\alpha/\abs{\beta})^{-1}},\qquad \alpha/\abs{\beta}= 1/p-1/q.
	\end{align} 
Here we have simply adapted the definition to our setting of anisotropic dilations.
For all $p,q,$ the Bloom weight is always defined as on the line \eqref{eq:bloom:weight}.
The following Theorem \ref{thm:lb} is an instant consequence of Theorem \ref{thm:osc}. 
\begin{thm}\label{thm:lb} Let $p,q\in(1,\infty)$ be arbitrary and $\mu\in A_{p,p}^{\gamma}(\R^2)$ and $\lambda\in A_{q,q}^{\gamma}(\R^2).$ Then,
	\begin{align}\label{eq:BloomLB}
		\Norm{b}{\BMO_{\nu}^{\gamma,\alpha}(\R^2)}\lesssim \Norm{[b,H_{\gamma}]}{L^p_{\mu}(\R^2)\to L^q_{\lambda}(\R^2)}.
	\end{align}
\end{thm}
Theorem \ref{thm:main:bdd} is then an instant consequence of Theorems \ref{thm:UPfrac} and \ref{thm:lb}. 
In particular, we conclude that $\BMO^{\gamma}$ is the correct condition on the symbol to characterize the boundedness of the commutator along all monomial curves in the plane, as up to now it was not evident, at least to the author, that the characterizing conditions would coalesce for the curves $\gamma_{\op{adj}}$ and $\gamma_{\op{opp}}.$
We move to discuss compactness in the case $q\geq p.$ 
\begin{defn} A linear operator $T:X\to Y$ between two Banach spaces is said to be compact, provided that for each bounded set $A\subset X,$ the image $TA\subset Y$ is relatively compact, i.e. the closure $\overline{TA}$ is compact in $Y$. 
\end{defn}
\begin{defn}\label{defn:VMOw} For a weight $w$ and a fractionality parameter $\alpha\in\R,$ define 
	$$
	\VMO_{w}^{\gamma,\alpha} = \VMO_{w}^{\beta, \alpha}\subset \BMO_{w}^{\beta, \alpha}
	$$ through the conditions
	\begin{align}\label{eq:VMOv1}
		\lim_{r\to 0}\sup_{\substack{Q\in\calR^{\beta} \\ \ell(Q)\leq r}}\calO_{w}^{\beta,\alpha}(b;Q) = 0,
	\end{align}
	\begin{align}\label{eq:VMOv2}
		\lim_{r\to \infty}\sup_{\substack{Q\in\calR^{\beta} \\ \ell(Q)\geq r}}\calO_{w}^{\beta,\alpha}(b;Q) = 0,
	\end{align}
	\begin{align}\label{eq:VMOv3}
		\sup_{Q\in\calR^{\beta}}\lim_{\abs{x}\to \infty}\calO_{w}^{\beta,\alpha}(b;Q+x) = 0.
	\end{align}
\end{defn}
Theorem \ref{thm:main:comp} follows from the following two Theorems \ref{thm:suff} and \ref{thm:nec} that address sufficiency and necessity. Naturally both are subject to the same limitations as in the case of boundedness.

\begin{thm}\label{thm:suff} Let $(1/p,1/q)\in\Xi(n)$ and $\alpha/\abs{\beta}=1/p-1/q.$ Then, $b\in\VMO^{\gamma,\alpha}(\R^n)$ implies that  $[b,H_{\gamma}]\in\calK(L^p_{\mu}(\R^n),L^q_{\lambda}(\R^n)).$
\end{thm}
In proving Theorem \ref{thm:suff}, one of the ingredients we argue for is the density of $C^{\infty}_c$ in $\VMO^{\beta,\alpha}.$ In practice we manage with just the density of $L^{\infty}_{\loc},$ and do not need the approximating element to be in $C^{\infty}_c.$ This being said,
let 
$$
\CMO^{\beta,\alpha} = \big\{f\in\BMO^{\beta,\alpha}\colon \exists\{g_j\}_j\subset C^{\infty}_c\mbox{ s.t. } \lim_{j\to\infty}\Norm{f-g_j}{\BMO^{\beta,\alpha
}}=0 \big\}
$$ be the closure of $C^{\infty}_c$ in $\BMO^{\beta,\alpha}.$ In proving Theorem \ref{thm:main:dens} below we show that the approximating element can be taken to be $C^{\infty}_c.$
\begin{thm}\label{thm:main:dens}  Let $\beta\in \R_+^n$ and $\alpha\geq 0.$ Then 
	$\CMO^{\beta,\alpha} \supset \VMO^{\beta,\alpha}.$ 
\end{thm} 
With $\beta = (1,\dots,1)$ and $\alpha= 0,$ Theorem \ref{thm:main:dens} is due to Uchiyama \cite{Uch1978}, and with $\beta = (1,\dots,1)$ and $\alpha>0,$ due to Guo, He, Wu and Yang \cite[Theorem 1.7.]{GuoHeWuYang2021}; in these classical cases the density of $C^{\infty}_c$ was a crucial ingredient.  
In this light, Theorem \ref{thm:main:dens} is hardly surprising; however, it is not so easy to prove and we consider it to be a real result of this article. Moreover, we expect that applications may be found for sufficiency of compactness where $\VMO^{\beta,\alpha}$ is by definition taken to be the sought-after condition, as an example, see Mair, Moen \cite{MaiMoe2023} and the references therein, where commutators are studied under bump conditions on the weights.
 
\begin{thm}\label{thm:nec} Let $\mu\in A_{p,p}^{\gamma}(\R^2)$ and $\lambda\in A_{q,q}^{\gamma}(\R^2).$ Then, $[b,H_{\gamma}]\in\calK(L^p_{\mu}(\R^2),L^q_{\lambda}(\R^2))$ implies that $b\in\VMO_{\nu}^{\gamma,\alpha}(\R^2).$
\end{thm}

Let us lastly move to the case $q<p.$ This case was recently settled for non-degenerate Calderón-Zygmund operators by Hyt\"{o}nen, Li, Tao, Yang \cite{HLTY2022},
\begin{align}\label{eq:HLTY}
 \Norm{[b,T]}{L^p(\R^n)\to L^q(\R^n)} \lesssim 1  \overset{*}{\Longleftrightarrow}	b\in \dot L^r(\R^n) \overset{**}{\Longrightarrow} [b,T]\in\calK(L^p,L^q),
\end{align}
by proving the $**$-implication, while the $*$-implication was proved by Hyt\"{o}nen \cite{HyLpLq} in the quantitative form 
\begin{align}\label{eq:q<pLr}
	\Norm{[b,T]}{L^p(\R^n)\to L^q(\R^n)}\sim\Norm{b}{\dot L^r}.
\end{align} 
Recall that in the current case, \eqref{eq:HLTY} implies the equivalence of compactness and boundedness. 
Above and below,
\begin{align}\label{eq:dotLr}
	\Norm{b}{\dot L^r(\R^n)}= \inf_{c\in\C}\Norm{b-c}{L^r(\R^n)},\qquad 1/q = 1/r + 1/p.
\end{align} 
We reproduce \eqref{eq:HLTY} in our setting through Theorems \ref{thm:main:q<pA} and \ref{thm:main:q<pB} below.
\begin{thm}\label{thm:main:q<pA} Let $q<p.$ Then, $b\in \dot L^r(\R^n)$ implies that $[b,H_{\gamma}]\in\calK(L^p(\R^n), L^q(\R^n)).$
\end{thm}
 We prove Theorem \ref{thm:main:q<pA} by essentially extrapolating from any of the cases $q \geq p.$ Our approach is new and we expect it to find applications in proving similar implications outside our specific setting; in this article we showcase a short proof of the $**$-implication on the line \eqref{eq:HLTY}, which is the main result of Hyt\"{o}nen et al. \cite[Theorem 2.5.]{HLTY2022}.

Recently H\"{a}nninen, Lorist and Sinko \cite[Theorem A]{HLS2023} settled the Bloom boundedness of commutators by providing the following estimate,
\begin{align}\label{eq:HLS}
	\Norm{[b,T]}{L^p(\mu)\to L^q(\lambda)}\sim \Norm{M^{\#}_{\nu} b}{L^r(\nu)}, \quad M^{\#}_{\nu} b = \sup_{Q\in\calR^{(1,\dots,1)}}\frac{1_Q}{\nu(Q)}\int_Q\abs{b-\ave{b}_Q},
\end{align}
valid for non-degenerate Calderón-Zygmund operators $T.$
In order to stay aligned with \eqref{eq:HLS}, we now treat the weights as measures, i.e. $\Norm{f}{L^s(w)} := \Norm{fw^{1/s}}{L^s},$ and this leads to the Bloom weight $\nu^{1/p+1/q'} = \mu^{1/p}\lambda^{-1/q}.$
By only using Theorem \ref{thm:osc}, we obtain from the methods of \cite{HLS2023} the following two-weight necessity result.
\begin{thm}\label{thm:main:q<pB} There holds that 
		\begin{align*}
			\Norm{M^{\gamma,\#}_{\nu} b}{L^r(\R^2(\nu))} \lesssim 	\Norm{[b,H_{\gamma}]}{L^p(\R^2(\mu))\to L^q(\R^2(\lambda))}, \quad M^{\gamma,\#}_{\nu} b = \sup_{Q\in\calR^{\gamma}}\frac{1_Q}{\nu(Q)}\int_Q\abs{b-\ave{b}_Q}.
				\end{align*}
\end{thm}
In the unweighted case $\Norm{M^{\gamma,\#}_{1} b}{\dot L^r}\sim \Norm{b}{\dot L^r},$
and thus Theorem \ref{thm:main:q<p} is implied by Theorems \ref{thm:main:q<pA} and \ref{thm:main:q<pB}. We should note however that Theorem \ref{thm:main:q<pB} is way too good for the unweighted lower bound $\Norm{b}{\dot L^r(\R^2)}\lesssim \Norm{[b,H_{\gamma}]}{L^p(\R^2)\to L^q(\R^2)}$ which is obtainable by Theorem \ref{thm:osc} and arguing as in \cite[Section 2.5.]{HyLpLq}.

\subsection{More research questions} 
Of the research Questions \ref{que:1} and \ref{que:2} above, we hope to address Question \ref{que:2} in a future work. 
There are some other natural gaps to be filled and we discuss these next.

We have made the decision to omit discussing any single-weight or Bloom-type two-weight sufficiency results.  
In particular, we expect some single-weight sufficiency results to follow through the sparse domination (obtained in this article) and the unweighted cases (obtained in this article) through a limited range extrapolation of compactness results such as Hyt\"{o}nen and Lappas \cite[Theorems 2.3. and 2.4.]{HytLap2023}, or Cao, Olivo, Yabuta \cite{CaoOliYab2022}. It would be very interesting to see this worked out in detail. What would be even more interesting would be to see Bloom-type sufficiency results.
\begin{que} Obtain Bloom-type two-weight sufficiency results for boundedness, and especially for compactness, for any arrangement $q<p,q=p,q>p$ of the exponents.
\end{que}

\begin{que}\label{que:3} Which results of this article can be reproduced with Hilbert transforms along curves with non-vanishing torsion, see e.g. \cite[Section 1.2.]{ClaOu2017} for the definition. 
\end{que}

One way to approach Question \ref{que:3} is to note that the machinery of Section \ref{sect:UB} is available for curves with non-vanishing torsion, thus in principle the sufficiency of testing conditions should be more straightforward. To address necessity, the only known approach at the moment is to obtain a version of Theorem \ref{thm:osc} for curves with non-vanishing torsion, here even the $n=2$ case would be interesting to see worked out.

\section{Sufficiency for $q\geq p$}\label{sect:UB}
\subsection{Boundedness}
Cladek, Ou \cite{ClaOu2017} discretized the Hilbert transform along a monomial curve as
\begin{align}\label{eq:Hsplit}
	H_{\gamma} = \sum_{\vec{j}\in \{0,1,2\}^n}\sum_{Q\in\calD^{\gamma}_{\vec{j}}}A_Q
\end{align}
to localized  dyadic pieces $\supp(A_Qf)\subset Q,$ see \cite[Section 2.]{ClaOu2017}.
\begin{lem}[\cite{ClaOu2017}, Lemma 2.3.]\label{lem:ClaOu2017A} 
	Let $\gamma:\R\to\R^n$ be a monomial curve, let $(1/r,1/s')\in\Omega(n).$ Let $Q_0\in\calD^{\gamma}_{\vec{j}}$ and let $\ch(Q_0)\subset\calD^{\gamma}_{\vec{j}}$ consist of those maximal dyadic $\gamma$-cubes $P$ 
	such that 
	\begin{align}
		\ave{f}_{P,r} \geq C\ave{f}_{Q_0,r}\, \text{ or } \,	\ave{g}_{P,s} \geq C\ave{g}_{Q_0,s}.
	\end{align}
	Then, for a choice of $C>1$ sufficiently large (independent of $Q_0$) there holds that 
	\begin{align*}
		\sum_{P\in\ch(Q_0)}\abs{P}\leq \frac{1}{2}\abs{Q_0},\qquad 	\Babs{\bave{\sum_{Q\in\mathcal{E}_{Q_0}} A_Qf,g}}\lesssim \abs{Q_0}	\ave{f}_{Q_0,r}\ave{g}_{Q_0,s},
	\end{align*}
where $\mathcal{E}_{Q_0} = \calD^{\gamma}_{\vec{j}}(Q_0)\setminus \cup_{P\in\ch(Q_0)}\calD^{\gamma}_{\vec{j}}(P).$
\end{lem} 
Behind Lemma \ref{lem:ClaOu2017A} was the following $L^p$ improving continuity estimate.
\begin{lem}[\cite{ClaOu2017}, Lemma 2.1.]\label{lem:ClaOu2017B} Let $\gamma$ be a monomial curve and $(1/r,1/s')\in \Omega(n).$ Then, there exists $\eta=\eta_{r,s}>0$ such that for all $\abs{y}\leq 1,$
	\begin{align}\label{eq:improving}
		\bNorm{A^{\gamma}_1-\tau_yA^{\gamma}_1}{L^r\to L^s}\lesssim \abs{y}^{\eta},
	\end{align}
	where $\tau_yf(x) = f(x-y)$ and
	\begin{align}\label{eq:A1}
	A_1^{\gamma}f(x) := \int_{\frac{1}{2}\leq \abs{t}<1}f(x-\gamma(t))\frac{\ud t}{t}.
	\end{align}
\end{lem}
Lemma \ref{lem:ClaOu2017B} in turn boiled down to the following $L^p$ improving estimate, see e.g. Christ \cite{Christ1998}, and Tao, Wright \cite{TaoWri2003}.
\begin{lem}\label{lem:ClaOu2017C} Let $\gamma$ be a monomial curve and $(1/r,1/s')\in \Omega(n).$ Then, there holds that 
	\begin{align}\label{eq:preimproving}
		\bNorm{A_1^{\gamma}}{L^r\to L^s}\lesssim_{r,s} 1.
	\end{align}
\end{lem}
For this section we only need Lemma \ref{lem:ClaOu2017A}, but when we address sufficiency of compactness we have direct use for Lemmas \ref{lem:ClaOu2017B} and \ref{lem:ClaOu2017C}.
\begin{defn}
	A collection of sets $\mathscr{S}$ is said to be sparse, if there exists a constant $\eta\in(0,1]$ and a collection of pairwise disjoint major subsets 
	$$
	\mathcal{E}_{\mathscr{S}} = \{E_P: E_P\subset P\in\mathscr{S},\, \abs{E_P}>\eta\abs{P}\}.
	$$ 
\end{defn}
By the splitting \eqref{eq:Hsplit}, Lemma \ref{lem:ClaOu2017A} and a stopping time argument, it is a standard routine, see e.g. the proof of \cite[Theorem 1.]{ClaOu2017}, to obtain a sparse form domination of the operator $H_{\gamma}$ in terms of $\gamma$-cubes.
It is by now an equally standard routine, see e.g. \cite{LOR1}, to obtain the sparse form domination of the commutator, i.e. to prove Theorem \ref{thm:SDOM}, we omit further details.
We have use for the following John-Nirenberg inequality adapted to our setting.
\begin{lem}\label{lem:JNfrac} Let $\alpha \geq 0$ and $1\leq s < \infty.$ Then, there holds that 
	\begin{align*}
	 	\Norm{b}{\BMO^{\beta,\alpha}_{1}} \sim \sup_{Q\in\calR^{\beta}} \abs{Q}^{-\alpha/\abs{\beta}}\left(\fint_Q\abs{b-\ave{b}_Q}^s\right)^{1/s}.
	\end{align*}
\end{lem}
\begin{proof} Standard proof by a stopping time argument in an anisotropic dyadic grid, see e.g. \cite[Equation (2.1)]{ClaOu2017} for a description of the grid and \cite[Appendix A. Theorem A.2.]{HOS2022} for the stopping time argument.
\end{proof}
\begin{fig}[h]
	\centering
	\includegraphics[scale=0.9]{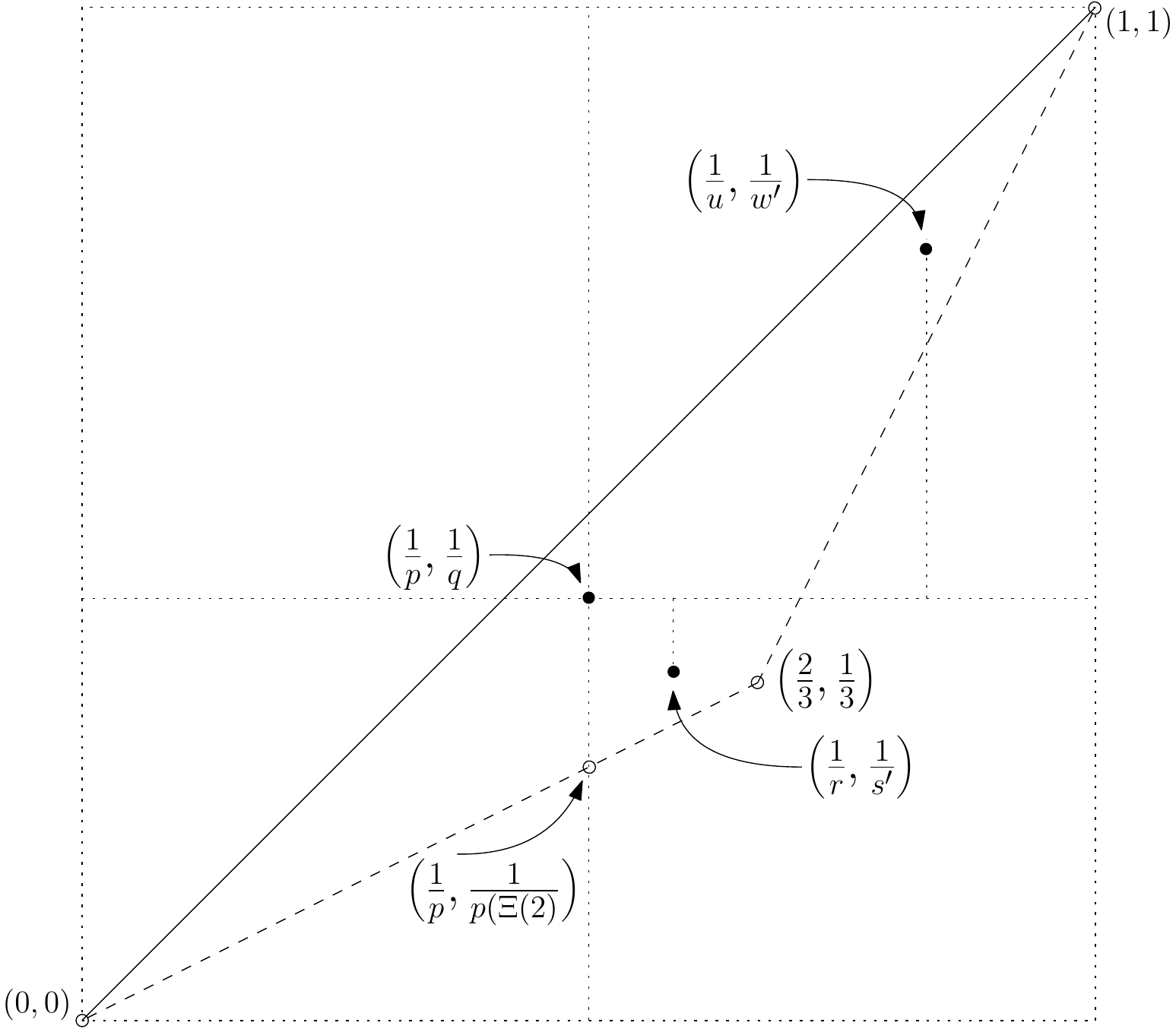}
	\caption{Showing $\Xi(2)$ and the choice of exponents $(1,r,1/s')$ in the proof of Theorem \ref{thm:UPfrac} and of $(1/u,1/w')$ in the proof of Lemma \ref{lem:pqcont}.
	} 
	\label{fig:ExpXi}
\end{fig}
\begin{proof}[Proof of Theorem \ref{thm:UPfrac}] 
	By the assumption
	$(1/p,1/q)\in \Xi(n) $
	let $(1/r,1/s')\in\Omega(n)$ be such that $1/p<1/r$ and $1/q>1/s',$ equivalently $r<p$ and $s<q',$ see Figure \ref{fig:ExpXi}.
By Proposition \ref{thm:SDOM} it is enough to bound $A_{b}^{r,s}[\mathscr{S}^{\gamma}]$ and $A_{b}^{r,s*}[\mathscr{S}^{\gamma}].$ Let $r<r_+<p$ and $1/r = 1/r_++1/r_+'.$ By H\"{o}lder's inequality and Lemma \ref{lem:JNfrac},
	\begin{align*}
		\bave{(b-\ave{b}_Q)f}_{Q,r}\leq \bave{b-\ave{b}_Q}_{Q,r_+^{'}}	\bave{f}_{Q,r_+}	\lesssim_{r_+^{'}}\Norm{b}{\BMO^{\gamma,\alpha}} \abs{Q}^{\alpha/\abs{\beta}} \bave{f}_{Q,r_+},
	\end{align*}
and thus 
\[
	A_{b}^{r,s}[\mathscr{S}^{\gamma}](f,g) \lesssim_{r_+^{'}}\Norm{b}{\BMO^{\gamma,\alpha}}  \sum_{Q\in \mathscr{S}^{\gamma}} \abs{Q}^{1+\alpha/\abs{\beta}} \bave{f}_{Q,r_+}\bave{g}_{Q,s}.
\]
	Then, using $1+\alpha/\abs{\beta} = 1/p+1/q',$ H\"{o}lder's inequality and $\Norm{\cdot}{\ell^q}\leq \Norm{\cdot}{\ell^p},$ we find
	\begin{align*}
	 \sum_{Q\in \mathscr{S}^{\gamma}} \abs{Q}^{1+\alpha/\abs{\beta}} \bave{f}_{Q,r_+}\bave{g}_{Q,s}
		&\leq \Big(  \sum_{Q\in \mathscr{S}^{\gamma}} \bave{f}_{Q,r_+}^{q}\abs{Q}^{\frac{q}{p}} \Big)^{\frac{1}{q}}\Big(  \sum_{Q\in \mathscr{S}^{\gamma}} \bave{g}_{Q,s}^{q'}\abs{Q} \Big)^{\frac{1}{q'}} \\ 
		&\leq  \Big(  \sum_{Q\in \mathscr{S}^{\gamma}} \bave{f}_{Q,r_+}^{p}\abs{Q} \Big)^{\frac{1}{p}}\Big(  \sum_{Q\in \mathscr{S}^{\gamma}} \bave{g}_{Q,s}^{q'}\abs{Q} \Big)^{\frac{1}{q'}} \\ 
		&\lesssim \Norm{M_{r_+}^{\gamma}}{L^p\to L^p}\Norm{M_s^{\gamma}}{L^{q'}\to L^{q'}}\cdot \Norm{f}{L^p}\Norm{g}{L^{q'}}.
	\end{align*}
	In the last estimate we used sparseness. Above  $M_u^{\gamma}f:= \sup_{Q\in\calR^{\gamma}}1_Q\ave{\psi}_{Q,u}$ and the finiteness of $\Norm{M_{r_+}^{\gamma}}{L^p\to L^p}$ and $\Norm{M_s^{\gamma}}{L^{q'}\to L^{q'}}$
	follow from $r_+<p$ and $s<q'.$ The estimate for $A_{b}^{r,s*}[\mathscr{S}^{\gamma}](f,g)$ is completely analogous.
\end{proof}


\subsection{Compactness}
In this section we couple the proof outline given in Lacey and Li \cite{LacLi2021} and detailed by us \cite{HOS2022} and prove Theorem \ref{thm:suff}. 
We show that
\begin{align}\label{eq:goal}
	H_{\gamma} = H_{\gamma,c} + H_{\gamma,\varepsilon},\qquad [b,H_{\gamma,c}]\in \calK(L^p,L^q),\qquad \lim_{\varepsilon\to 0}\Norm{	[b,H_{\gamma,\varepsilon}]}{L^p\to L^q} = 0.
\end{align}
Then we are done by the fact that compact linear operators form a closed subspace of all bounded linear operators. 

Given any function $\psi$ and $r>0,$ define
\begin{align}\label{eq:bump0} 
	\psi_x^{0,r}(y) = \psi\big(r^{-\beta}(x-y)\big)
\end{align}
and for $0<r\leq R<\infty$ partition unity as
\begin{align*}
	1 = \psi_x^{0,r} +  \psi_x^{r,R} +  \psi_x^{R,\infty},\qquad  \psi_x^{r,R} :=  \psi_x^{0,R}-  \psi_x^{0,r},\qquad \psi_x^{R,\infty} := 1 - \psi_x^{0,R}.
\end{align*}
We denote $\psi^{r,R} := \psi_0^{r,R}.$ 
Now, let
\begin{align}\label{eq:bump}
	\varphi\in C^1(\R^n;[0,1]),\qquad 	\varphi(x) = \begin{cases}
		1,\quad \abs{x} \leq\frac{1}{2}, \\ 
		0,\quad \abs{x}\geq 1.
	\end{cases}
\end{align}
Then, we split
\begin{align*}
	H_{\gamma} &= \big(\varphi^{0,R}+\varphi^{R,\infty}\big)H_{\gamma} \big(\varphi^{0,R}+\varphi^{R,\infty}\big) \\ 
	&= \varphi^{0,R}H_{\gamma} \varphi^{0,R} 
	+ \left(  \varphi^{R,\infty}H_{\gamma} \varphi^{R,\infty}
	+ \varphi^{R,\infty}H_{\gamma} \varphi^{0,R}
	+ \varphi^{0,R}H_{\gamma} \varphi^{R,\infty}  \right) =  
	H_{\gamma}^{0,R} + H_{\gamma}^{R,\infty}.
\end{align*}
The bracketed term $H_{\gamma}^{R,\infty}$ is good and the other we split further,
\begin{align*}
	H_{\gamma}^{0,R} = 	H_{\gamma}^{0,R}\big( \chi_x^{0,r} + 	\chi_x^{r,cR} +	\chi_x^{cR,\infty} \big) = 	H_{\gamma}^{0,R}\chi_x^{0,r} + 	H_{\gamma}^{0,R}\chi_x^{r,cR},
\end{align*}
where 
\begin{align}\label{eq:chi}
	\chi_x^{a,b} := (1_{B(0,1)})_x^{a,b},
\end{align} 
and we chose  $c\gtrsim_{\beta}1$ to be so large that $H_{\gamma}^{0,R}\chi_x^{cR,\infty}  = 0.$ 
Indeed, by abusing notation, if 
$$
H_{\gamma}^{0,R}\chi_x^{cR,\infty}f(x)=  \varphi^{0,R}(x)H_{\gamma}\left( \varphi^{0,R}(y)\chi_x^{cR,\infty}(y)\right)f(x)\not=0,
$$ then $\abs{x}_{\beta},\abs{y}_{\beta}\lesssim R,$ but the integral is non-zero iff $\norm{x-y}{\beta}\gtrsim cR,$ which by the quasi triangle inequality is impossible, provided that $c$ is large.
In total, we have 
\begin{align}\label{eq:split}
	H_{\gamma} =  H_{\gamma,c} + H_{\gamma,\varepsilon},\qquad  H_{\gamma,c} = 	H_{\gamma}^{0,R}\chi_x^{r,cR},\qquad H_{\gamma,\varepsilon} = 	H_{\gamma}^{0,R}\chi_x^{0,r} + H_{\gamma}^{R,\infty},
\end{align}
where we now take on the convention of demanding that $\varepsilon \sim r \sim R^{-1}.$

Next, we formulate local $L^p$-to-$L^q$  boundedness and continuity estimates.
\begin{lem}\label{lem:pqcont} Let $(1/p,1/q)\in \Xi(n).$ Let $Q=[-1,1]^{n}.$ Then, there exists an absolute constant $c_{\gamma}>0$ so that
	\begin{align}\label{eq:pqbdd}
		\bNorm{A^{\gamma}_{1}}{L^p(c_{\gamma}Q)\to L^q(Q)}\lesssim_{p,q} 1.
	\end{align}	
	Moreover, there exists $\eta=\eta_{p,q}>0$ such that for all $\abs{y}\leq 1$ there holds that 
	\begin{align}\label{eq:pqimproving}
		\bNorm{A^{\gamma}_1-\tau_yA^{\gamma}_1}{L^p(c_{\gamma}Q)\to L^q(Q)}\lesssim_{p,q} \abs{y}^{\eta}.
	\end{align}
\end{lem}
\begin{proof} 
	We check \eqref{eq:pqimproving}. Pick a point $(1/u,1/w')\in\Omega(n)$ so close to $(1,1)$ that both $u<p$ and $q<w$ hold, see Figure \ref{fig:ExpXi}. By H\"{o}lder's inequality in the first and last steps (recall that $\abs{Q}\sim 1$) and Lemma \ref{lem:ClaOu2017B} we find
	\begin{align*}
		\Norm{( A^{\gamma}_1-\tau_yA^{\gamma}_1)f}{L^q(Q)} &\lesssim \Norm{( A^{\gamma}_1-\tau_yA^{\gamma}_1)f}{L^w(Q)}=\Norm{( A^{\gamma}_1-\tau_yA^{\gamma}_1)(f1_{c_{\gamma}Q})}{L^w(Q)}\\ 
		&\leq \Norm{( A^{\gamma}_1-\tau_yA^{\gamma}_1)(f1_{c_{\gamma}Q})}{L^w} \lesssim_{w,u} \abs{y}^{\eta}\Norm{f1_{c_{\gamma}Q}}{L^u} \lesssim \abs{y}^{\eta}\Norm{f}{L^p(c_{\gamma}Q)}.
	\end{align*}
	The claim \eqref{eq:pqbdd} is obtained analogously, now using Lemma \ref{lem:ClaOu2017C}.
\end{proof}
The bounds \eqref{eq:pqbdd} and \eqref{eq:pqimproving} correspond to equiboundedness and equicontinuity.
\begin{lem}[Fréchet-Kolmogorov]\label{lem:FreKol} Let $q\in (0,\infty).$ Then $\calF\subset L^q$ is relatively compact if and only if $\calF$ is
	\begin{enumerate}
		\item equibounded: 
		\[		
		\sup_{f\in\calF}\Norm{f}{L^q} \lesssim 1,
		\]
		\item equicontinuous: 
		\[
		\lim_{\abs{y}\to 0}\sup_{f\in\calF}\Norm{f-\tau_y f}{L^q} =0,
		\]
		\item and equivanishing: 
		\[
		\lim_{A\to\infty}\sup_{f\in\calF}\Norm{1_{B(0,A)^c}f}{L^q} = 0.
		\]
	\end{enumerate}
\end{lem}

\begin{prop}\label{prop:Hc} Let $(1/p,1/q)\in\Xi(n)$ and $b\in L^{\infty}_{\loc}.$ Then 
	$[b,	H_{\gamma,c}] \in \calK(L^p,L^q).$
\end{prop}

\begin{proof}	
	Directly from the assumptions, the multipliers $\varphi^{0,R},$ $\supp(\varphi^{0,R})b$ and $\varphi^{0,R}b$
	are $L^s$-to-$L^s$ bounded, for all $s>0.$ Thus, writing out
	\begin{align*}
		[b,	H_{\gamma,c}]  = b\varphi^{0,R}H_{\gamma,c}\chi_{x}^{r,cR}\varphi^{0,R} - \varphi^{0,R}H_{\gamma,c}\chi_{x}^{r,cR}\varphi^{0,R} b
	\end{align*} 
	and using the fact that a sufficient condition for $S\circ S'$ to be compact is that both $S,S'$ are bounded and at least  one is compact,
	it is enough to show that 
	$$
	\widetilde{H}_{\gamma} = \varphi^{0,R}H_{\gamma}\chi_{x}^{r,cR}\in \calK(L^p,L^q).
	$$
	
	Let there be a family $\calF = \{f_j\}_j$ such that $\sup_j\Norm{f_j}{L^p}\lesssim1,$ and by Lemma \ref{lem:FreKol} we show that $\widetilde{H}_{\gamma}\calF\subset L^q$ is equicontinuous, equibounded and equivanishing.
	
	Denote $Q= R^{\beta}[-1,1]^n\supset \supp(\varphi^{0,R})$ and suppose without loss of generality that $r = 2^{-n}$ and $cR=2^{n}$ are dyadic. By splitting
	\begin{align*}
		\widetilde{H}_{\gamma} = 	\sum_{-n+1\leq i \leq n}\widetilde{H}_{\gamma}^i,\qquad \widetilde{H}_{\gamma}^i := \varphi^{0,R}H_{\gamma}\chi_{x}^{2^{i-1},2^{i}},
	\end{align*}
	it is enough to show that each $\widetilde{H}_{\gamma}^i\calF$  is equicontinuous, equibounded and equivanishing, since these conditions are preserved under linear combinations.
	
	The equivanishing condition
	\[
	\lim_{A\to\infty}\sup_{j}\bNorm{1_{B(0,A)^c}	\widetilde{H}_{\gamma}^if_j }{L^q} = 0,
	\]
	clearly holds when $A$ is so large that $\supp(\varphi^{0,R})\subset B(0,A),$ as then $1_{B(0,A)^c}	\widetilde{H}_{\gamma}^if_j=0.$	
	
	For equiboundedness we have
	\begin{equation}\label{eq:Hc:equi}
		\begin{split}
			\bNorm{\widetilde{H}_{\gamma}^if_j }{L^q} \leq 	\bNorm{H_{\gamma}\chi_{x}^{2^{i-1},2^{i}}f_j }{L^q(Q)}
			\overset{*}{\lesssim}_{p,q,\gamma,R} \bNorm{f_j}{L^p(c_{\gamma}^{\beta}Q)}\leq \bNorm{f_j}{L^p},
		\end{split}
	\end{equation}
	where the $*$ estimate follows from the line \eqref{eq:pqbdd} of Lemma \ref{lem:pqcont} by rescaling everything to the $\gamma$-cube $Q.$ We give full details. Denote $\lambda := 2^i$ and write
	\begin{align*}
		H_{\gamma}\chi_{x}^{2^{i-1},2^{i}}f_j(x) = A^{\gamma}_1(f\circ \lambda^{\beta})\circ\lambda^{-\beta}.
	\end{align*}
	Then,
	\begin{align*}
		 	\bNorm{H_{\gamma}\chi_{x}^{2^{i-1},2^{i}}f_j }{L^q(Q)} &=  	\bNorm{A^{\gamma}_1(f_j\circ \lambda^{\beta})\circ\lambda^{-\beta}}{L^q(Q)} \\ 
		 	&= \lambda^{\abs{\beta}/q}	\bNorm{A^{\gamma}_1(f_j\circ \lambda^{\beta})}{L^q(\lambda^{-\beta}_0Q)} \lesssim_{q,\gamma,R}\bNorm{A^{\gamma}_1(f_j\circ \lambda^{\beta})}{L^q(\lambda^{-\beta}_0Q)}.
	\end{align*}
Write $\lambda^{-\beta}_0Q =\cup_{m}Q_m$ as a finitely overlapping union of translates of $[-1,1]^{n}.$ By the translation invariance of $A^{\gamma}_1$ and \eqref{eq:pqbdd} of Lemma \ref{lem:pqcont} we obtain 
	\begin{align*}
		\bNorm{A^{\gamma}_1(f_j\circ \lambda^{\beta})}{L^q(\lambda^{-\beta}_0Q)}  &\sim_q \sum_m 	\bNorm{A^{\gamma}_1(f_j\circ \lambda^{\beta})}{L^q(Q_m)}\lesssim \sum_m \bNorm{f_j\circ \lambda^{\beta} }{L^p(c_{\gamma}^{\beta}Q_m)} \\ 
		&=  \sum_m c_{\gamma}^{\abs{\beta}/p} \bNorm{f_j\circ \lambda^{\beta} \circ c_{\gamma}^{\beta}}{L^p(Q_m)}\sim_{p,\gamma} \bNorm{f_j\circ \lambda^{\beta} \circ c_{\gamma}^{\beta}}{L^p(\lambda^{-\beta}_0Q)} \\
		&\sim_{p,\gamma}\bNorm{f_j\circ c_{\gamma}^{\beta}}{L^p(Q)}\sim_{p,\gamma} \bNorm{f_j}{L^p(c_{\gamma}^{\beta}Q)}
	\end{align*}
and thus \eqref{eq:Hc:equi} is established.

	To check equicontinuity, we show that
	\begin{align}\label{eq:kred1}
		\bNorm{(\widetilde{H}_{\gamma}^i -\tau_y\widetilde{H}_{\gamma}^i)f_j}{L^q}\lesssim_{p,q,\gamma,\varphi,R} \max\left(\abs{y}, \abs{y}_{\beta}^{\eta_{p,q}} \right)\Norm{f_j}{L^p}
	\end{align} 
for all sufficiently small $\abs{y}.$ Split, 
	\begin{align*}
		\widetilde{H}_{\gamma}^i -\tau_y	\widetilde{H}_{\gamma}^i  &= \left(\varphi^{0,R}-\tau_y\varphi^{0,R}\right) H_{\gamma}\chi_{x}^{2^{i-1},2^i} + \tau_y\varphi^{0,R}\left( H_{\gamma}\chi_{x}^{2^{i-1},2^i} - \tau_yH_{\gamma}\chi_{x}^{2^{i-1},2^i} \right) \\ 
		&= I + II.
	\end{align*}
	By \eqref{eq:Hc:equi} and $\varphi^{0,R}\in C^1_c,$
	we find
	\begin{equation}\label{eq:ec1}
		\begin{split}
			\Norm{If_j}{L^q} &\leq \Norm{\varphi^{0,R}-\tau_y\varphi^{0,R}}{L^{\infty}}\bNorm{H_{\gamma}\chi_{x}^{2^{i-1},2^{i}}f_j }{L^q(Q)} \\
			&\lesssim_{p,q,\gamma,R} \Norm{\varphi^{0,R}-\tau_y\varphi^{0,R}}{L^{\infty}}\Norm{f}{L^p}\lesssim_{\varphi}  \abs{y}\Norm{f_j}{L^p}.
		\end{split}
	\end{equation}
	Let $\abs{y}_{\beta}<\ell(Q).$ 
	Then by \eqref{eq:pqimproving} of Lemma \ref{lem:pqcont} we obtain
	\begin{equation}\label{eq:ec2}
		\begin{split}
			\Norm{IIf_j}{L^q} 
			&\leq \Norm{\tau_y\varphi^{0,R}}{L^{\infty}}\bNorm{(H_{\gamma}\chi_{x}^{2^{i-1},2^{i}} - \tau_yH_{\gamma}\chi_{x}^{2^{i-1},2^{i}})f_j}{L^q(Q)} \\ 
			 &\overset{**}{\lesssim}_{p,q,\gamma,R}  \abs{y}^{\eta_{p,q}}_{\beta}\Norm{f_j}{L^p(c_{\gamma}^{\beta}Q)} \leq \abs{y}^{\eta_{p,q}}_{\beta}\Norm{f_j}{L^p},
		\end{split}
	\end{equation}
where the $**$-estimate follows similarly as in the case of equiboundedness, this time using \eqref{eq:pqimproving} of Lemma \ref{lem:pqcont}.
	Combining \eqref{eq:ec2} and \eqref{eq:ec1}, the bound \eqref{eq:kred1} follows.
\end{proof}

The proof of the following Proposition \ref{prop:err:sprs} is completely analogous to the proof of \cite[Proposition 4.7.]{HOS2022}, the only difference being that we have a sparse collection of $\beta$-cubes instead of ordinary cubes, we omit the proof.
\begin{prop}\label{prop:err:sprs}
	Let 
	$(1/p,1/q)\in \Xi(n)$ and $(1/r,1/s')\in\Omega(n)$ be such that $r<p$ and $s<q'.$ Let  $\scrS\subset\calR^{\gamma}$ be sparse and denote
	\begin{align}\label{eq:split:sparse}
		\scrS_{k} = \scrS\setminus \scrS_{k}^0 ,\qquad  \scrS_k^0 = \Big\{ Q\in\scrS: \ell(Q)\in [k^{-1},k] ,\ \dist(Q,0)\leq k \Big\},\qquad k>0.
	\end{align}
	Then, given any $\varepsilon>0,$ there exists a large $k=k_{\varepsilon}>0$ such that  
	$$
	\mathcal{A}_{b}^{r,s}[\mathscr{S}_k](f,g) +	 \mathcal{A}_{b}^{r,s*}[\mathscr{S}_k](f,g)
	\lesssim \vare \Norm{f}{L^p}\Norm{g}{L^{q'}}.
	$$
\end{prop}

\begin{prop}\label{prop:He} Let $(1/p,1/q)\in\Xi(n),$ let $\alpha/\abs{\beta} = 1/p-1/q$ and $b\in\VMO^{\gamma,\alpha}.$ Then,
	\begin{align*}
		\lim_{\varepsilon\to 0}\bNorm{[b,	 H_{\gamma,\varepsilon}]}{L^{p}\to L^q} = 0.
	\end{align*}
\end{prop}
\begin{proof} 
	It is enough to show that for each fixed $k>0,$ the following holds: for all $\varepsilon > 0$ small enough (depending on $k$) there exists a sparse collection $\scrS\subset\calR^{\gamma}$ such that
	\begin{align}\label{eq:step1}
		\babs{\bave{[b,H_{\gamma,\varepsilon}]f , g}}\lesssim  \mathcal{A}_{b}^{r,s}[\mathscr{S}](f,g) + \mathcal{A}_{b}^{r,s*}[\mathscr{S}](f,g),\qquad \scrS = \scrS_k,
	\end{align}
	where $r,s$ and $\scrS_k$ are as in Proposition \ref{prop:err:sprs}.
As 
	\begin{align}\label{eq:step2}
		H_{\gamma,\varepsilon} = 	\varphi^{0,R}	H_{\gamma}	\varphi^{0,R}\chi_x^{0,r} + \left(   \varphi^{R,\infty}	H_{\gamma} \varphi^{R,\infty}
		+ \varphi^{R,\infty}	H_{\gamma} \varphi^{0,R}
		+ \varphi^{0,R}	H_{\gamma}\varphi^{R,\infty} \right),
	\end{align}
	it is enough to show that \eqref{eq:step1} is satisfied for each of the above four pieces. The bracketed pieces are handled identically. For example, directly from the sparse form domination of $[b, H_{\gamma}],$ recall Theorem \ref{thm:SDOM}, we acquire a sparse collection $\scrS$ such that
	\begin{align}\label{step1}
		\babs{\bave{ [b, \varphi^{0,R}H_{\gamma} \varphi^{R,\infty}]f ,g}} \lesssim  \mathcal{A}_{b}^{r,s}[\mathscr{S}]\left( f\varphi^{R,\infty} , g\varphi^{0,R} \right) +  \mathcal{A}_{b}^{r,s*}[\mathscr{S}]\left(f\varphi^{R,\infty} , g\varphi^{0,R}\right).
	\end{align}
	Considering the first term on the right-hand side of \eqref{step1}, we have 
	\begin{align*}
		\mathcal{A}_{b}^{r,s}[\mathscr{S}]\left( f\varphi^{R,\infty} , g\varphi^{0,R} \right)\leq \mathcal{A}_{b}^{r,s}[\mathscr{S}^{R,\infty}] (f,g),
	\end{align*}
	where $\mathscr{S}^{R,\infty} = \{ Q\in\scrS: Q\cap B_{\beta}(0,\frac{1}{2} R)^c \neq\emptyset \}.$ Clearly $\mathscr{S}^{R,\infty} = \scrS_k^{R,\infty}$ for sufficiently small $\varepsilon$ (large $R$). The other term is completely analogous.
	
	It remains to handle the term $\bave{ [b, \varphi^{0,R}H_{\gamma}	\varphi^{0,R}\chi_x^{0,r}]f ,g},$ which we rewrite as 
	\[
	\bave{ [b, H_{\gamma}\chi_x^{0,r}]\widetilde{f} ,\widetilde{g}},\qquad \wt{f} = \varphi^{0,R} f,\quad \wt{g} = \varphi^{0,R} g.
	\]  
	We express $\R^n = \bigcup_j Q_j$ as a disjoint union of $\gamma$-cubes $Q_j\in\calR^{\gamma}$ with $\ell(Q_j) = r.$
	Then,
	\begin{align*}
		\bave{ [b,  H_{\gamma}\chi_x^{0,r}]\widetilde{f} ,\widetilde{g}} =  \sum_j \bave{ 1_{Q_j}[b,  H_{\gamma}\chi_x^{0,r}]\widetilde{f} ,\widetilde{g}},
	\end{align*}
	and it is enough to show that each $\bave{ 1_{Q_j}[b,  H_{\gamma}\chi_x^{0,r}]\widetilde{f} ,\widetilde{g}}$ admits a sparse domination approximately localized to $Q_j.$ There exists an absolute constant $C=C_{\gamma}>0$ independent of $r$ such that
	\[
	1_{Q_j}(x) H_{\gamma}\chi_x^{0,r}f(x) = 1_{Q_j}(x) H_{\gamma}\chi_x^{0,r}(1_{Q_j^*}f)(x),\qquad Q_j^* = C^{\beta}Q_j.
	\]
	Thus we can write
	\begin{align*}
		\bave{ 1_{Q_j}[b, H_{\gamma}\chi_x^{0,r}]\widetilde{f} ,\widetilde{g}} = \bave{ 1_{Q_j}[b, H_{\gamma}\chi_x^{0,r}](1_{Q_j^*}\widetilde{f}) ,\widetilde{g}}.
	\end{align*}
	By the standard proof of the sparse domination of the commutator, using only Lemma \ref{lem:ClaOu2017A}, we obtain
	\begin{align}\label{eq:step3}
		\babs{\bave{ 1_{Q_j}\big[b,  H_{\gamma}\chi_x^{0,r}\big](1_{Q_j^*}\widetilde{f}) ,\widetilde{g}}} \lesssim   \mathcal{A}_{b}^{r,s}[\mathscr{S}(Q_j^{**}) ]\big( f,g\big) +   \mathcal{A}_{b}^{r,s*}[ \mathscr{S}(Q_j^{**})]\big( f,g\big),
	\end{align}
	where $\cup\scrS(Q_j^{**})\subset Q_j^{**} := (C')^{\beta}Q_j^*$ for some absolute constant $C'>0,$ and the implicit constant in \eqref{eq:step3} is independent of $r.$ All in all, we have now shown that
	\begin{align*}
		\babs{\bave{ [b, \varphi^{0,R}H_{\gamma}	\varphi^{0,R}\chi_x^{0,r}]f ,g}} \lesssim  \mathcal{A}_{b}^{r,s}[\mathscr{S} ]\big( f,g\big) +   \mathcal{A}_{b}^{r,s*}[ \mathscr{S}]\big( f,g\big),\quad \mathscr{S} :=  \cup_j\mathscr{S}(Q_j^{**}).
	\end{align*}
	By $\ell(Q_j^{**}) \sim_{\gamma} r,$ for all $j,$ it follows that $\mathscr{S} = \mathscr{S}_k$ for a choice of $\varepsilon = r$ sufficiently small.
\end{proof}

\begin{proof}[Proof of Theorem \ref{thm:suff}]  By Theorem \ref{thm:main:dens}, let $b_\delta \in L^{\infty}_{\loc}$ be such that $\Norm{b-b_{\vare}}{\BMO^{\gamma,\alpha}}\leq \delta.$ Then by Theorem \ref{thm:UPfrac} we know that $\Norm{[b-b_{\delta}, H_{\gamma}]}{L^p\to L^q}\lesssim \delta.$ In other words, $[b_{\delta}, H_{\gamma}]$ approximates $[b, H_{\gamma}]$ to arbitrary precision as $\delta\to 0.$ Thus it is enough to show that $[b_{\delta},H_{\gamma}]\in \calK(L^p,L^q).$ By \eqref{eq:split} we split $H_{\gamma} = H_{\gamma,c}+H_{\gamma,\varepsilon},$ and then by Propositions \ref{prop:Hc} and \ref{prop:He} $[b_{\delta},H_{\gamma,c}]$ is a compact operator that approximates $[b_\delta,H_\gamma]$ to arbitrary precision as $\vare\to 0.$
\end{proof}

\section{Density of $C^{\infty}_c$ in $\VMO^{\beta,\alpha}$}\label{sect:density} 
Above we used the fact that locally bounded functions are dense in $\VMO^{\beta,\alpha},$ but actually much more is true and in this section we prove Theorem \ref{thm:main:dens}. This being said, it is much easier to obtain the density of $L^{\infty}_{\loc}$ than of $C^{\infty}_c,$ and the reader interested only in the first, only has to read up to Remark \ref{rem:dens:easy}. 
When $\beta= 1=(1,\dots,1),$ the brief history of $\CMO^{1,\alpha}$ and $\VMO^{1,\alpha}$ is recorded as the following two theorems.
\begin{thm}[Uchiyama \cite{Uch1978}]\label{thm:uchi} There holds that  $\CMO^{1,0}(\R^n) = \VMO^{1,0}(\R^n).$ 
\end{thm}
\begin{thm}[Guo et al. \cite{GuoHeWuYang2021}]\label{thm:GuoDensity}  There holds that $\CMO^{1,\alpha}(\R^n) = \VMO^{1,\alpha}(\R^n),$ for $0<\alpha<1.$ Moreover, if $\alpha\geq 1,$ then $\VMO^{1,\alpha}$ consists of the constant functions.
\end{thm}

	The approximation constructed by Uchiyama \cite{Uch1978} for $\CMO^{1,0}\supset \VMO^{1,0}$ works in the case $\CMO^{\beta,0} \supset \VMO^{\beta,0}$ by simply replacing cubes with $\beta$-cubes and we omit further details. Uchiyama's approximation is discontinuous and does not work when $\alpha>0$ due to the fact that functions in $\BMO^{\alpha,\beta}$ are continuous and a discontinuity blows up the semi-norm $\Norm{\cdot}{\BMO^{\beta,\alpha}}$
	on small scales, see Lemma \ref{lem:H\"{o}lderBMOa} below. 
	For this reason, Guo et al. \cite{GuoHeWuYang2021} constructed a non-trivial new approximation for $\CMO^{1,\alpha}\supset \VMO^{1,\alpha}.$ We do not know how to extend the approximation of Guo et al. to our situation and hence we provide new approximations. 
	
	First, we reprove Theorem \ref{thm:GuoDensity} by giving an alternative approximation that uses only basic extension, truncation and mollification; the drawback is that this only works under the assumption that $\abs{\cdot}_{\beta}^{\alpha}$ has an equivalent metric, which is not in general available for arbitrary $\beta\in\R_+^n$ and $\alpha>0.$ Necessarily, we provide another approximation that works without any restrictions on $\beta,\alpha,$ see Proposition \ref{prop:density3}. 
	
	\begin{rem}
			The validity of the reverse inclusion $\CMO^{\beta,\alpha} \subset \VMO^{\beta,\alpha}$  depends on the interplay of the quasi metric $\abs{\cdot}_{\beta}$ and the fractionality parameter $\alpha.$ We go into this direction just enough to obtain  
				\begin{align}\label{eq:CMObeta=VMObeta}
				\CMO^{\beta,0} \subset \VMO^{\beta,0},
			\end{align}
			which we use when proving Theorem \ref{thm:main:q<pA}.
		
		Let $f\in\BMO^{\beta,\alpha}$ and let $g\in C^{1}_c$ be an approximation $\Norm{f-g}{\BMO^{\beta,\alpha}}<\varepsilon.$ Then, for all cubes $Q,$
		\begin{align}\label{eq:discussion}
			\calO^{\beta,\alpha}(f;Q)\lesssim 	\calO^{\beta,\alpha}(f-g;Q) + 	\calO^{\beta,\alpha}(g;Q) \lesssim \varepsilon +	\calO^{\beta,\alpha}(g;Q).
		\end{align}
		Then, $\calO^{\beta,\alpha}(g;Q)\to 0$ in the large \eqref{eq:VMOv2} and far away \eqref{eq:VMOv3} cases directly from $g\in L^{\infty}_c.$  With the small case \eqref{eq:VMOv1}, using $g\in C^1_c,$ we estimate
		\begin{align}\label{sharp1}
			\mathcal{O}^{\beta,\alpha}(g;Q)&\leq \Norm{\nabla g}{L^{\infty}}\abs{Q}^{-\alpha/\abs{\beta}} \fint_Q\fint_Q\abs{x-y}\ud x\ud y \lesssim_g \diam(Q)\abs{Q}^{-\alpha/\abs{\beta}}.
		\end{align}
		If $\beta = 1,$ thus $Q\in\calR^{(1,\dots,1)},$ the right-hand side of \eqref{sharp1} is comparable to $\ell(Q)^{1-\alpha}$ which vanishes when $\ell(Q)\to 0,$ by $\alpha<1,$ thus it follows that $\CMO^{\alpha}\subset \VMO^{\alpha},$  for $0\leq \alpha<1.$ 
	 	Given arbitrary $\beta\in\R_+^n,$ thus $Q\in\calR^{\beta},$ and if $\alpha = 0,$ then the right-hand side of \eqref{sharp1} is comparable to $\ell(Q)$ which vanishes when $\ell(Q)\to 0,$ thus we obtain \eqref{eq:CMObeta=VMObeta}.
		
		Lastly, suppose that $\beta = (1,\dots,1)$ and $\alpha \geq 1$ and that $f$ satisfies the condition \eqref{eq:VMOv1}, which implies by Lemma \ref{lem:H\"{o}lderBMOa} below that
		\begin{align}\label{same1}
			\lim_{l\to 0}\sup_{\substack{x\not=y \\ \abs{x-y}\leq l}}\frac{\abs{f(x)-f(y)}}{\abs{x-y}} = 0.
		\end{align}
		From \eqref{same1}, the gradient is zero,
		hence $\VMO^{1,\alpha}$ consists of the constant functions and $\VMO^{1,\alpha}\not=\CMO^{1,\alpha},$ as remarked in Guo et al. \cite{GuoHeWuYang2021}.
	\end{rem}

\begin{lem}\label{lem:H\"{o}lderBMOa} Let $\beta\in\R_+^n$ and $\alpha>0.$ Then there exists a constant $C=C_{\beta,\alpha}>0$ such that for all $B_0\in\calR^{\beta}$ and $l>0$ there holds that  
	\begin{align}\label{eq:scaleloc1}
		\Norm{b}{\BMO^{\beta,\alpha}_l(B_0)} 	 \lesssim \Norm{b}{\dot C^{0,\beta,\alpha}_l(B_0)} \lesssim	\Norm{b}{\BMO^{\beta,\alpha}_l(C^{\beta}B_0)},
	\end{align}
	where for any localization $A\subset \R^n$ and a scale $l>0,$ we have defined
	\begin{align}\label{eq:scaleloc2}
		\Norm{b}{\dot C^{0,\beta,\alpha}_l(A)} := \sup_{\substack{x,y\in A \\ 0<\abs{x-y}_{\beta}\leq l}}\frac{\abs{b(x)-b(y)}}{\abs{x-y}_{\beta}^{\alpha}},\qquad
		\Norm{b}{\BMO^{\beta,\alpha}_l(A)} := \sup_{\substack{Q\in\calR^{\beta} \\ Q\subset A \\ \ell(Q)\leq l}}\calO^{\alpha}(b;Q).
	\end{align}
	If either of the constraints $l,A$ above is not present we drop that symbol and in this case 
	\begin{align*}
		\Norm{b}{\BMO^{\beta,\alpha}}\sim \Norm{b}{\dot C^{0,\beta,\alpha}}.
	\end{align*}
\end{lem}
\begin{rem}\label{rem:dens:easy} We remark at this point that Uchiyama's approximation is $L^{\infty}_{\loc},$ and so are continuous functions (in particular by Lemma \ref{lem:H\"{o}lderBMOa} all functions in $\BMO^{\beta,\alpha}$). Thus we have now shown that $L^{\infty}_{\loc}$ is dense in $\VMO^{\beta,\alpha},$ for arbitrary data $\beta\in\R_ +^n$ and $\alpha\geq 0.$ Recall that this was enough to conclude the proof of Theorem \ref{thm:nec}. In other words, the reader not interested in obtaining the density of $C^{\infty}_c$ may skip to the next Section \ref{sect:LB}.
\end{rem}

\begin{proof}[Proof of Lemma \ref{lem:H\"{o}lderBMOa}] The proof with the scale constraint $\leq l$ and a localization to a top $\beta$-ball $B_0$ is identical to the proof without a localization. The proof of arbitrary $\beta\in\R_+^n$ is identical to the case $\beta = (1,\dots,1),$ for which we refer to N. Meyers \cite{Meyers1964}.
\end{proof}
	\begin{proof}[Proof of Theorem \ref{thm:GuoDensity}]  
		Without loss of generality suppose that $f$ is real-valued.
		By the conditions \eqref{eq:VMOv1} and \eqref{eq:VMOv2} let $i\leq k$ be such that 
		\begin{align}\label{eq:dena}
			\calO^{\alpha}(f;Q) \leq \varepsilon' \ell(Q),\qquad \mbox{ if } \qquad \ell(Q)\in (2^i,2^k)^c,
		\end{align}
		especially $\calO^{\alpha}(f;Q) > \varepsilon'$ possibly only for cubes $Q$ with $\ell(Q)\in (2^i,2^k)$ and it follows from the conditions \eqref{eq:VMOv1},  \eqref{eq:VMOv2} and  \eqref{eq:VMOv3} 
		that there exists $R>2^k$ so that 
		\begin{align}\label{eq:denb}
			\calO^{\alpha}(f;Q)\leq\varepsilon',\qquad \mbox{if} \qquad Q\cap Q_R^c \not= \emptyset,\qquad \mbox{where} \qquad Q_R := (-R,R)^n.
		\end{align}
		Actually a small argument is needed to conclude \eqref{eq:denb}, see e.g. the proof of \cite[Theorem A.4.]{HOS2022}.
		 By choosing $\varepsilon'$ small enough and reformulating \eqref{eq:dena} and \eqref{eq:denb} through Lemma \ref{lem:H\"{o}lderBMOa} we obtain 
		\begin{align}\label{eq:de2}
			\Norm{f}{\dot C^{0,1,\alpha}(Q_{R}^c)} + 	\Norm{f}{\dot C^{0,1,\alpha}_{2^i}} \leq \vare,
		\end{align}
	for any prescribed $\varepsilon>0.$

		Then, define
		\begin{align}\label{eq:ext1}
			f_{R}^e(x) = 1_{Q_R}(x)f(x) + 1_{Q_R^c}(x)\inf_{y\in \partial Q_R}\{f(y)+\varepsilon\abs{x-y}^{\alpha}\},
		\end{align}
		which extends $f1_{Q_R}.$
		Next we check that
		\begin{align}\label{eq:de3}
			\Norm{f_R^e}{\dot C^{0,1,\alpha}(Q_R^c)}\leq \varepsilon
		\end{align}
		and
		\begin{align}\label{eq:de4}
			\Norm{f_R^e}{\dot C^{0,1,\alpha}_{2^i}(\R^n)}\leq \varepsilon.
		\end{align}
		
		The boundary being closed, the infimum is achieved as a minimum and for each $x\in Q_R^c$ we specify one $y_x\in \partial Q_R$ so that 
		\begin{align}\label{minimizer}
			f_{R}^e(x) = f(y_x)+\varepsilon\abs{x-y_x}^{\alpha}.
		\end{align}	
		Let $x,z\in Q_R^c.$ Without loss of generality we assume that $f_R^e(x)\geq f_R^e(z)$ and then using triangle inequality and $\alpha\leq 1$ we obtain 
		\begin{align*}
			f_R^e(x) - f_R^e(z) &\leq f(y_x)+\vare\abs{x-y_x}^{\alpha} - (f(y_x)+\vare\abs{z-y_x}^{\alpha}) = \vare\left( \abs{x-y_x}^{\alpha} -\abs{z-y_x}^{\alpha}  \right) \\
			 &\leq \vare\left( \abs{x-z}^{\alpha}+\abs{z-y_x}^{\alpha} -\abs{z-y_x}^{\alpha}  \right)\leq \vare\abs{x-z}^{\alpha}.
		\end{align*}
		Thus \eqref{eq:de3} is checked.
		Clearly \eqref{eq:de4} follows from \eqref{eq:de3} and
		\[
		1_{\partial Q_R}f = 1_{\partial Q_R}f_K^e,\qquad \Norm{f}{\dot C^{0,1,\alpha}_{2^i}(Q_R)}\leq\vare
		\]
		of which only the left remains to be checked, the right being immediate from \eqref{eq:de2}.
		If $x\in\partial Q_R,$ then from the definition of infimum
		\[
		f_R^e(x) =  f(y_x)+\varepsilon\abs{x-y_x}^{\alpha}\leq f(x)+\varepsilon\abs{x-x}^{\alpha} = f(x).
		\]
		Thus by $\Norm{f}{\dot C^{0,\beta,\alpha}(\partial Q_R)}\leq \Norm{f}{\dot C^{0,\beta,\alpha}(Q_R^c)}\leq\varepsilon$ we obtain
		\[
		0\leq f(x)-f_R^e(x) = f(x)-f(y_x) - \varepsilon\abs{x-y_x}_{\beta}^{\alpha}\leq (\Norm{f}{\dot C^{0,\beta,\alpha}(\partial Q_R)}-\varepsilon)\abs{x-y_x}^{\alpha}\leq 0.
		\]
				
		Next we obtain a compact support by truncating $f_R^e.$ From that $f$ is bounded on $Q_R$ we find some $M >> R$ large enough and a point $c_M\in Q_{2M}\setminus Q_M$ so that 
		\begin{align}\label{eq:de5}
			\sup_{ Q_M}f_R^e < f_R^e(c_M) < \inf_{ \R^n\setminus Q_{2M}}f_R^e.
		\end{align}
		Then, consider the sphere-like set
		\[
		\mathfrak{S} = \{x\in\R^n: f_R^e(x) = f(c_M)\}\subset Q_{2M}\setminus Q_M;
		\]
		by the intermediate value theorem $\{tx:t>0\}\cap \mathfrak{S}\not=\emptyset$ for each $x\in\R^n;$ moreover $\mathfrak{S}$ is connected from the continuity of $f.$
		We let  $x\in \mathfrak{S}_{\op{in}}$ if there exists a path from $x$ to the origin not intersecting $\mathfrak{S}.$
		Then we truncate
		\begin{align}\label{eq:de7}
			f_{\varepsilon} = 1_{\mathfrak{S}_{\op{in}}}f_R^e + 1_{\mathfrak{S}_{\op{out}}}f_R^e(c_M)
		\end{align}
		so that $f_{\varepsilon}$ is constant on  $\mathfrak{S}_{\op{out}}=\R^n\setminus \mathfrak{S}_{\op{in}},$ and clearly, for all $l>0$ and  $A\subset\R^n,$ we have 
		\begin{align}\label{eq:de8}
			\Norm{f_{\varepsilon}}{\dot C^{0,1,\alpha}_l(A)}\leq \Norm{f_R^e}{\dot C^{0,1,\alpha}_l(A)}.
		\end{align}
	
		Next we show that  
		\begin{align}\label{eq:de9}
			\Norm{f-f_{\varepsilon}}{\dot C^{0,1,\alpha}}\lesssim\varepsilon
		\end{align}
	by checking cases.
		If $x,y\in Q_R,$ then from that $f_{\varepsilon}=f_R^e=f$ on $Q_R\subset Q_M\subset \mathfrak{S}_{\op{in}},$ this case is cleared.
		If $x,y\in Q_R^c,$ then by \eqref{eq:de2}, \eqref{eq:de3} and \eqref{eq:de8} we have 
		\begin{align}\label{eq:de10}
			\Norm{f_{\varepsilon}}{\dot C^{0,1,\alpha}(Q_R^c)}\leq\Norm{f_R^e}{\dot C^{0,\beta,\alpha}(Q_R^c)} \leq	\Norm{f}{\dot C^{0,\beta,\alpha}(Q_R^c)}\leq \varepsilon.
		\end{align}
		Thus this case is   cleared.	
		Lastly, if $x\in Q_R$ and $y\in Q_R^c,$ then using that $f=f_{\vare}$ on $Q_R\cup\partial Q_R$ and picking the point $z\in\partial Q_R\cap \{tx+(1-t)y : t\in [0,1]\},$ we find
		\begin{align*}
			\abs{(f-f_{\varepsilon})(x) - (f-f_{\varepsilon})(y)} &= \abs{(f-f_{\varepsilon})(z) - (f-f_{\varepsilon})(y)} \\
			&\lesssim\big(\Norm{f}{\dot C^{0,\beta,\alpha}(Q_R^c)} + \Norm{f_{\varepsilon}}{\dot C^{0,\beta,\alpha}(Q_R^c)}\big)  \abs{z-y}^{\alpha} \lesssim \varepsilon\abs{x-y}^{\alpha}.
		\end{align*}
		We have now shown that $f$ can be approximated by a continuous compactly supported functions. It does not take much more to upgrade $C_c$ to $C^{\infty}_c,$ which is what we do next. 
		
		Let
		\begin{align}\label{eq:mollify1}
		h_{\varepsilon}(x) = \int_{\R^n} \eta_{2^i}(y)f_{\varepsilon}(x-y)\ud y,\qquad	0\leq \eta_{2^i}\in C^{\infty}_c(B(0,2^{i-1})),\qquad \Norm{\eta_{2^i}}{L^1} = 1,
		\end{align}
		so that clearly $h_{\varepsilon}\in C^{\infty}_c.$
		We have  
		\begin{align*}
			\abs{h_{\vare}(x)-h_{\vare}(y)}= \Babs{\int_{\R^n}\eta_{2^i}(z)\left(f_{\varepsilon}(x-z)-f_{\varepsilon}(y-z)\right) \ud z} \leq  \Norm{\eta_{2^i}}{L^1}\Norm{f_{\varepsilon}}{\dot C^{0,1,\alpha}_{\abs{x-y}}}\abs{x-y}^{\alpha},
		\end{align*}
		which shows that $\Norm{h_{\varepsilon}}{\dot C^{0,1,\alpha}_{l}}\leq \Norm{f_{\varepsilon}}{{\dot C^{0,1,\alpha}_{l}}}$ on every scale $l,$ and in particular that
		\begin{align}\label{eq:de11}
			\Norm{h_{\varepsilon}}{\dot C^{0,1,\alpha}_{2^i}}\leq \Norm{f_{\varepsilon}}{{\dot C^{0,1,\alpha}_{2^i}}}\leq \Norm{f_R^e}{{\dot C^{0,1,\alpha}_{2^i}}}\leq \varepsilon.
		\end{align}
		By Lemma \ref{lem:H\"{o}lderBMOa}, we obtain
		\begin{align}\label{eq:scalessmall}
			\sup_{\ell(Q)\leq 2^i}\calO^{\alpha}(f_{\varepsilon}-h_{\varepsilon};Q) \leq \sup_{\ell(Q)\leq 2^i}\calO^{\alpha}(f_{\varepsilon};Q) + \sup_{\ell(Q)\leq 2^i}\calO^{\alpha}(h_{\varepsilon};Q) \lesssim \varepsilon.
		\end{align}
		Then, suppose that $\ell(Q) > 2^i$ and write $Q = \cup_{j}Q_j$ as a finitely overlapping union of cubes with $\ell(Q_j)=2^{i-1}.$ Then,
		\begin{equation}\label{eq:scaleslarge}
			\begin{split}
				\ell(Q)^{\alpha+n}\calO^{\alpha}(f_{\varepsilon}-h_{\varepsilon};Q) &\lesssim \sum_j\int_{Q_j}\abs{f_{\varepsilon}(x)-h_{\varepsilon}(x)}\ud x \\ 
				&\leq \sum_j\int_{Q_j}\int_{2Q_ j}\eta_{2^i}(y)\abs{f_{\varepsilon}(x)-f_{\varepsilon}(x-y)}\ud y \ud x \\
				&\lesssim \Norm{f_{\varepsilon}}{\dot C^{0,1,\alpha}_{2^i}}\sum_j\ell(Q_j)^{\alpha+n}\lesssim \varepsilon 	\ell(Q)^{\alpha+n},
			\end{split}
		\end{equation}
		which shows that $\sup_{\ell(Q) >  2^i}\calO^{\alpha}(f_{\varepsilon}-h_{\varepsilon};Q)\lesssim \varepsilon.$  Taking into account   \eqref{eq:scalessmall} we have now shown that $\Norm{f_{\varepsilon}-h_{\varepsilon}}{\BMO^{1,\alpha}}\lesssim \varepsilon,$ which concludes the proof.
	\end{proof}

Recalling that the case $\alpha = 0$ is clear, to prove Theorem \ref{thm:main:dens} it remains to prove the following proposition.
\begin{prop}\label{prop:density3}
	Let $\beta\in\R_+^n$ and $\alpha > 0.$ Let $f\in\VMO^{\beta,\alpha}.$ Then for each $\vare>0$ there exists $g\in C^{\infty}_c$ such that  
	\begin{align}\label{eq:approx}
		\Norm{f-g}{\BMO^{\beta,\alpha}}\leq\varepsilon.
	\end{align}
\end{prop}


\begin{proof}[Proof of Proposition \ref{prop:density3}] It follows from the conditions \eqref{eq:VMOv1},  \eqref{eq:VMOv2} and  \eqref{eq:VMOv3} that we find $0<i<<R$ such that 
	\begin{align}\label{eq:denaa}
		\frac{\abs{f(x)-f(z)}}{\abs{x-z}_{\beta}^{\alpha}}\leq \varepsilon,\quad \mbox{ if }\quad \abs{x-z}_{\beta} \leq i
	\end{align}
and
	\begin{align}\label{eq:dencc}
		\sup_{x\in B_{\beta}(0,R)^c}\sup_{y\in\R^n}\frac{\abs{f(x)-f(y)}}{\abs{x-y}_{\beta}^{\alpha}}\leq \vare.
	\end{align}
	Pick a bump
	\begin{align*}
		0\leq \eta_{i} \in C^{\infty}(B_{\beta}(0,i)),\qquad  \Norm{\eta_{i}}{L^1} = 1.
	\end{align*}
	Consider the following function
	\begin{align*}
		\tau:\R^n\to B_{\beta}(0,M),\qquad	\tau(x) = \begin{cases}
			x,\qquad &\abs{x}_{\beta}<M, \\
		\left(	\frac{2M-\abs{x}_{\beta}}{M}\right)^{\beta}x,\qquad M\leq &\abs{x}_{\beta}<2M,\\
			0,\qquad &\abs{x}_{\beta}\geq 2M.
		\end{cases}
	\end{align*}
There holds that
\begin{align}\label{eq:tau1}
	\abs{\tau(x)-\tau(z)}_{\beta} \lesssim \abs{x-z}_{\beta},
\end{align}
and if $M>>R$ is chosen sufficiently large, that
	\begin{align}\label{eq:tau3}
		\abs{\tau(x)-\tau(z)}_{\beta} \lesssim M^{-1}\abs{x-z}_{\beta},\quad\mbox{if}\quad\tau(x),\tau(z)\in B_{\beta}(0,2R),\quad x \in B_{\beta}(0,M)^c.
	\end{align}
	The properties \eqref{eq:tau1} and \eqref{eq:tau3} are clear when $\beta = (1,\dots,1)$ and similarly checked for arbitrary $\beta\in \R_+^n,$ we omit the details.
	Now, let $\tau\in C_c$ be any function satisfying the properties \eqref{eq:tau1} and \eqref{eq:tau3},
	then, we define
	\begin{align}\label{eq:funct:approx}
		g = (f*\eta_{i})\circ\tau \in C_c.
	\end{align}
	
	We first show that there exists an absolute constant $c>0$ such that 
	\begin{align}\label{eq:dens1}
		\Norm{g}{\dot C^{0,\beta,\alpha}_{c i}} \lesssim \varepsilon.
	\end{align}
	By \eqref{eq:tau1} there exists $c^{-1}>0$ such that
	$
	\abs{\tau(x)-\tau(z)}_{\beta}\leq c^{-1}\abs{x-z}_{\beta},
	$
	so let $\abs{x-z}_{\beta}\leq ci$ and obtain  $\abs{\tau(x)-\tau(z)}_{\beta}\leq i.$
	Thus by \eqref{eq:denaa} and \eqref{eq:dencc} we find
	\begin{equation}\label{eq:eq1}
		\begin{split}
			\abs{g(x)-g(z)} &\leq \int \abs{\eta_{i}(y)}\abs{f(\tau(x)-y)-f(\tau(z)-y)}\ud y \\ 
			&\qquad\lesssim \Norm{f}{\dot C^{0,\beta,\alpha}_{i}}\Norm{\eta_{i}}{L^1}\abs{\tau(x)-\tau(z)}_ {\beta}^{\alpha}\lesssim \varepsilon \abs{x-z}_ {\beta}^{\alpha},
		\end{split}
	\end{equation}
	and \eqref{eq:dens1} is checked.
	
	Then we show that 
	\begin{align}\label{eq:dens2}
		\sup_{x\in B_{\beta}(0,M)^c}\sup_{z\in\R^n}\frac{\abs{g(x)-g(z)}}{\abs{x-z}_{\beta}^{\alpha}}\leq \vare.
	\end{align}
	
	Let $x\in B_{\beta}(0,M)^c$ be fixed and we do cases. If $\tau(x)\in B_{\beta}(0,2R)^c,$ then $\tau(x)-y\in B_{\beta}(0,R)^c,$ for each $y\in\supp(\eta_{i}),$ by $i<<R.$ Thus by \eqref{eq:dencc}, we have 
	\begin{align*}
		\abs{g(x)-g(z)} &\leq \int \abs{\eta_{i}(y)}\abs{f(\tau(x)-y)-f(\tau(z)-y)}\ud y \\  
		&\leq 	\left(\sup_{u\in B_{\beta}(0,R)^c}\sup_{v\in\R^n}\frac{\abs{f(u)-f(v)}}{\abs{u-v}_{\beta}^{\alpha}}\right)\abs{\tau(x)-\tau(z)}_{\beta}^{\alpha}\lesssim \varepsilon\abs{x-z}_{\beta}^{\alpha},
	\end{align*}
	where in the last estimate we used \eqref{eq:tau1}.
	The symmetrical case  $\tau(z)\in B_{\beta}(0,2R)^c$ is identical. 
	Then consider $\tau(x),\tau(z)\in B_{\beta}(0,2R).$ By \eqref{eq:tau3} we obtain
	\begin{equation}\label{eq:dens3}
		\begin{split}
			\abs{g(x)-g(z)} &\leq \int \abs{\eta_{i}(y)}\abs{f(\tau(x)-y)-f(\tau(z)-y)}\ud y \\ 
			&\lesssim \Norm{f}{\dot C^{0,\beta,\alpha}(B_{\beta}(0,3R))}\abs{\tau(x)-\tau(z)}_{\beta}^{\alpha} \\ 
			&\lesssim  M^{-1}  \Norm{f}{\dot C^{0,\beta,\alpha}(B_{\beta}(0,3R))}\abs{x-z}_{\beta}^{\alpha} \lesssim \vare\abs{x-z}_{\beta}^{\alpha},
		\end{split}
	\end{equation}
	where in the last estimate we demanded $M>>R$ to be sufficiently large.
	Now we fix such an $M>>R$ that both \eqref{eq:dens1} and \eqref{eq:dens2} hold.

	Then, we show that
	\[
	\Norm{f-g}{\dot C^{0,\beta,\alpha}}\lesssim \varepsilon.
	\] 
	We do cases. Small scales follow immediately from the estimate
	\begin{align}\label{eq:scales:small}
		\Norm{f-g}{\dot C^{0,\beta,\alpha}_{i}}\leq \Norm{f}{\dot C^{0,\beta,\alpha}_{i}} + \Norm{g}{\dot C^{0,\beta,\alpha}_{i}} \lesssim \Norm{f}{\dot C^{0,\beta,\alpha}_{i}} + \Norm{g}{\dot C^{0,\beta,\alpha}_{ci}}\lesssim \varepsilon,
	\end{align}
	true by \eqref{eq:denaa} and \eqref{eq:dens1}. 
	Then, suppose that $\abs{x-z}_{\beta}>i$ and $x,z\in B_{\beta}(0,\frac{M}{C}),$ where $C>1$ is the absolute constant from Lemma \ref{lem:H\"{o}lderBMOa}. We will show that
	\[
	\abs{(f-g)(x)-(f-g)(z)}\lesssim \varepsilon\abs{x-z}_{\beta}^{\alpha}.
	\]
	By Lemma \ref{lem:H\"{o}lderBMOa} it is enough to show that for all $\beta$-balls $B\subset B_{\beta}(0,M),$ there holds that $\calO^{\alpha}(f-g;B)\lesssim\varepsilon.$ So fix such $B.$ If $\op{rad}_{\beta}(B)\lesssim i,$ then $\calO^{\alpha}(f-g;B)\lesssim\varepsilon$ follows directly from \eqref{eq:scalessmall}. If $\op{rad}_{\beta}(B)\gtrsim i,$ let there be a finitely overlapping covering $B\subset \cup_j B_j$ with $\beta$-balls $B_j$ of $\op{rad}_{\beta}(B_j)= i$ and such that $\abs{\cup_jB_j}\lesssim \abs{B}.$ For each single $B_j$ we have 
	\begin{align*}
		\int_{B_j}\abs{f-g} &\leq \int_{B_j}\int\eta_i(y)\abs{f(x)-f(x-y)}\ud y\ud x \\ 
		&\lesssim \int_{B_j}\fint_{2^{\beta}B_j}\abs{f(x)-f(x-y)}\ud y\ud x \lesssim \Norm{f}{\dot C^{0,\beta,\alpha}_{2i}}\abs{B_j}^{1+\alpha/\abs{\beta}} \lesssim \varepsilon\abs{B_j}^{1+\alpha/\abs{\beta}},
	\end{align*}
	and thus in total,
	\begin{align*}
		\abs{B}^{1+\alpha/\abs{\beta}}\calO^{\alpha}(f-g;B)&\lesssim \sum_{j}\int_{B_j}\abs{f-g} \lesssim \varepsilon\sum_{j}\abs{B_j}^{1+\alpha/\abs{\beta}} \lesssim  \varepsilon \abs{B}^{1+\alpha/\abs{\beta}}.
	\end{align*}
	Finally, we check that
	\begin{align}
		\sup_{x\in B_{\beta}(0,\frac{M}{C})^c} \sup_{z\in\R^n} \frac{\abs{(f-g)(x)-(f-g)(z)}}{\abs{x-z}_{\beta}^{\alpha}}\lesssim \vare.
	\end{align} 
	We let $M$ be so big that both \eqref{eq:denaa} and \eqref{eq:dens2} hold with $M/C$ in place of $M,$ which is possible by the fact that $C$ is an absolute constant. Then, we obtain 
	\begin{align*}
		&\sup_{x\in B_{\beta}(0,\frac{M}{C})^c}  \sup_{z\in\R^n}  \frac{\abs{(f-g)(x)-(f-g)(z)}}{\abs{x-z}_{\beta}^{\alpha}} \\ 
		&\qquad\leq 	\sup_{x\in B_{\beta}(0,\frac{M}{C})^c}  \sup_{z\in\R^n}  \frac{\abs{f(x)-f(z)}}{\abs{x-z}_{\beta}^{\alpha}}+	\sup_{x\in B_{\beta}(0,\frac{M}{C})^c}  \sup_{z\in\R^n}  \frac{\abs{g(x)-g(z)}}{\abs{x-z}_{\beta}^{\alpha}}\lesssim \vare.
	\end{align*}
	We have now shown that $f$ can be approximated by $g\in C_c.$ 
	
	To upgrade $C_c$ to $C^{\infty}_c$ mollify $g*\eta_i$ and argue similarly as around and after the line \eqref{eq:mollify1} to show that $g*\eta_i\in C^{\infty}_c$ approximates $g;$ alternatively one can modify $\tau$ to achieve $\tau\in C^{\infty}_c$ to obtain $g\in C^{\infty}_c$ as early as \eqref{eq:funct:approx}.
\end{proof}

\section{Awf in the plane}\label{sect:LB}
In this section we remark how to extend the awf argument to those monomial curves that intersect opposite quadrants of the plane.
We emphasize that this section is not independent as such and should be treated as a companion to \cite[Section 3.]{Oik2022}.

\subsection{Outline} We begin by describing what we want to achieve.
Let $Q\in\calR^{\gamma}$ be fixed and write
\begin{align}\label{eq:dualize}
	\int_Q\abs{b-\ave{b}_Q} = \int bf,\qquad 1_Qf = f,\quad\abs{f}\leq 2,\quad \int f= 0.
\end{align}
For now let $W_1,P\subset \R^2$ be unspecified sets and $g_S$ a generic function supported on the set $S\in\{Q,W_1,P\}.$ We formally expand 
	\begin{align}\label{A1}
	f  = \left[h_QH_{\gamma}^*g_{W_1} - g_{W_1} H_{\gamma}h_Q\right] + \left[ h_{W_1} H_{\gamma}^*g_P - g_PH_{\gamma}h_{W_1}\right]  + \wt{f}_P,
\end{align}  
where
\begin{align}\label{A2}
	h_Q = \frac{f}{H_{\gamma}^*g_{W_1}},\qquad h_{W_1} = \frac{g_{W_1} H_{\gamma}h_Q}{H_{\gamma}^*g_P},\qquad \wt{f}_P = g_PH_{\gamma}\Big( \frac{g_{W_1}}{H_{\gamma}^*g_P}H_{\gamma}\big( \frac{f}{H_{\gamma}^*g_{W_1}}\big)\Big).
\end{align} 
The major part of the proof of fixing the above unspecified data so that 
\begin{align}\label{A3}
	\abs{h_Q} \lesssim_{A}\abs{f},\qquad \abs{g_{W_1}}\lesssim 1_{Q_1},\qquad 	\abs{h_{W_1}}\lesssim_{A} \Norm{f}{\infty} 1_{W_1},	
\end{align}
and moreover so that for any fixed $\vare>0,$
\begin{align}\label{A33}
	 \abs{\wt{f}_P}\lesssim  \varepsilon\Norm{f}{\infty}1_P,\qquad\int_P\wt{f}_P = 0.
\end{align}
Once we obtain \eqref{A3} and \eqref{A33} we are almost done.
We will see that $P\in\calR^{\gamma},$ and by \eqref{A33} the function $\wt{f}_P$ satisfies the same properties on $P$ as $f$ did on $Q$ (and in addition has the decay factor of $\varepsilon$) compare with \eqref{eq:dualize}. Then, we find a set $W_2\not=W_1$ and expand
		\begin{align}\label{A5}
		\wt{f}_P  = \left[h_PH_{\gamma}g_{W_2} - g_{W_2} H_{\gamma}^*h_P\right] + \left[ h_{W_2} H_{\gamma}g_Q - g_QH_{\gamma}^*h_{W_2}\right]  + \wt{(\wt{f}_P)}_Q,
	\end{align} 
and we have
\begin{align}\label{A6}
	\abs{h_P} \lesssim_{A} \abs{\wt{f}_P},\qquad \abs{g_{W_2}}\lesssim 1_{W_2},\qquad 	\abs{h_{W_2}}\lesssim_{A} \Norm{f}{\infty} 1_{W_2},
\end{align}
\begin{align}\label{A66}
	 \babs{\wt{(\wt{f}_P)}_Q}\lesssim  \varepsilon\Norm{\wt{f}_P}{\infty}1_Q,\qquad\int_Q \wt{(\wt{f}_P)}_Q = 0.
\end{align}
	Here is one difference to the proof in \cite{Oik2022}, where $W_1=W_2,$ see \cite[Section 3.1.]{Oik2022}. 
Finally, it will be clear from the construction of the sets, defined below in Section \ref{sec:geometry}, that 
\begin{align}\label{A7}
	\abs{Q}\sim\abs{P}\sim\abs{W_1}\sim\abs{W_2}.
\end{align}

\subsection{Geometry}\label{sec:geometry}
Without loss of generality we specify $\beta = (1,2)$ and end up with parabolic curves
\begin{align*}
	\gamma(t) = 
	\begin{cases}
		(\varepsilon_1\abs{t},\varepsilon_2\abs{t}^{2}),\quad t > 0, \\ 
		(\delta_1\abs{t},\delta_2\abs{t}^{2}),\quad t\leq 0,
	\end{cases}
\end{align*} 
and moreover specify to the two cases $\gamma\in \{\gamma_{\op{adj}},\gamma_{\op{opp}} \}$
as defined on the line \eqref{eq:adjopp}.

We now fix $Q=I\times J\in\calR^{(1,2)}$ and start working on a scale comparable to $\ell(Q)= \ell(I).$ We define two auxiliary intervals
\begin{align}\label{eq:auxintervals}
	I_{A_i} = \ell(I)[A_i,A_i+N_i],\qquad i=1,2,
\end{align}
and for the parameters specify
\begin{align}\label{eq:parameters}
 N_1  = 3,\qquad N_2 = 3,\qquad 	A_2 = 10A_1.
\end{align}
Then, define 
\begin{fig}[h]
	\centering
	\includegraphics[scale=0.9]{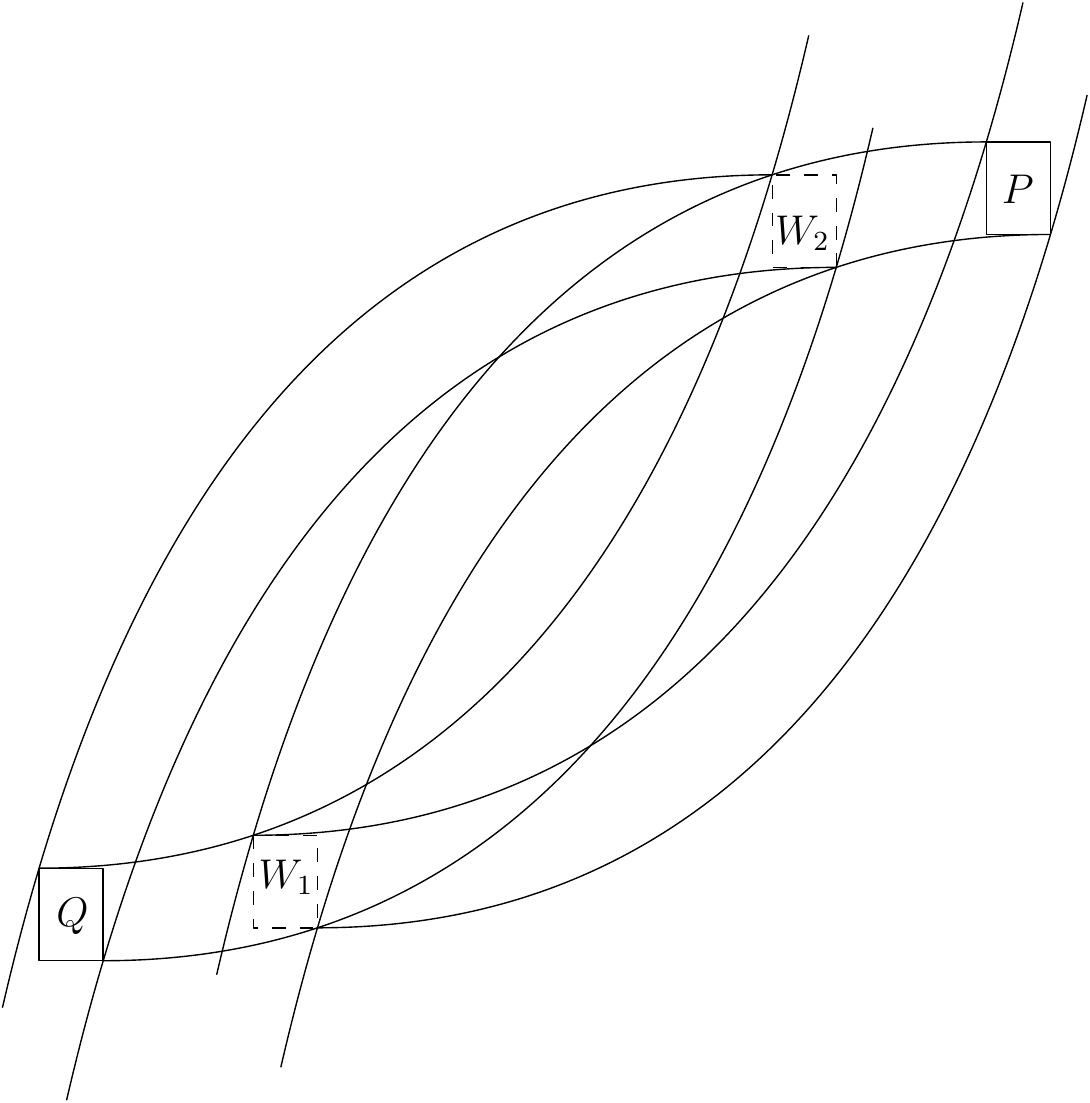}
	\caption{Showing some of the sets \eqref{eq:geometry1}. The set $V_1$ is the dashed box inside $W_1.$ The set $W_2$ is constructed by backtracing the process.
	} 
	\label{fig:geom}
\end{fig}
\begin{equation}\label{eq:geometry1}
	\begin{split}
		\begin{cases}
			&Q, \\ 
			&\wt{Q} = \big\{ x+\gamma(t): x\in Q,\, t\in I_{A_1}\big\}, \\
			&P =  V_1 + \gamma(c_{I_{A_2}}),\quad V_1 = Q +\gamma(c_{I_{A_1}}), \\
			&\wt{P} = \big\{ z-\gamma(t): z\in P,\, t\in I_{A_2}\big\},\quad \\
			&W_1 = \wt{Q}\cap \wt{S};
		\end{cases}
	\end{split}
\end{equation}
for a sketch see Figures \ref{fig:geom} and \ref{fig:key}.
We notate
\begin{align*}
 	I(a,\pm, B) &= \{t\in\R : a \pm \gamma(t)\in B\}\subset \R, \\
 	\phi(a,\pm,B) &= \{a\pm\gamma(t)\in B: t\in \R\} \subset B. 
\end{align*}
The variable $a\in\R^2$ is the reference point, the sign $\pm\in\{-,+\}$ indicates direction, while the last variable is a set $B\subset \R^2.$
\begin{lem}[Key-lemma]\label{lem:key} Let $z\in P$ be arbitrary. Then,
	\begin{align*}
		\bigcup_{y\in\phi(z,-,W_1)}\phi(y,-,Q) = Q.
	\end{align*}
\end{lem}
\begin{proof} Consider Figure \ref{fig:key}. Let $z\in P$ be arbitrary. The parabolic segment $\phi(z,-,W_1)$ contains a point from the top $y_t$ and bottom $y_b$ parts of the boundary of $W_1;$ that such points $y_t,y_b$ exist follows from $A_2>>A_1,$ see the dashed tangents and how $W_1$ is deeply embedded inside $\wt{Q}$ (the dashed area enclosing $W_1$). 
 Now $Q$ lies entirely between the segments $\phi(y_t,-,\mathbb{R}^2)$ and $\phi(y_b,-,\mathbb{R}^2).$ In particular, the parabolic segments $\phi(y,-,\mathbb{R}^2)\cap Q = \phi(y,-,Q)$ that begin from points $y\in\phi(z,-,W_1)$ foliate $Q.$
\end{proof}
\begin{fig}[h]
	\centering
	\includegraphics[scale=0.9]{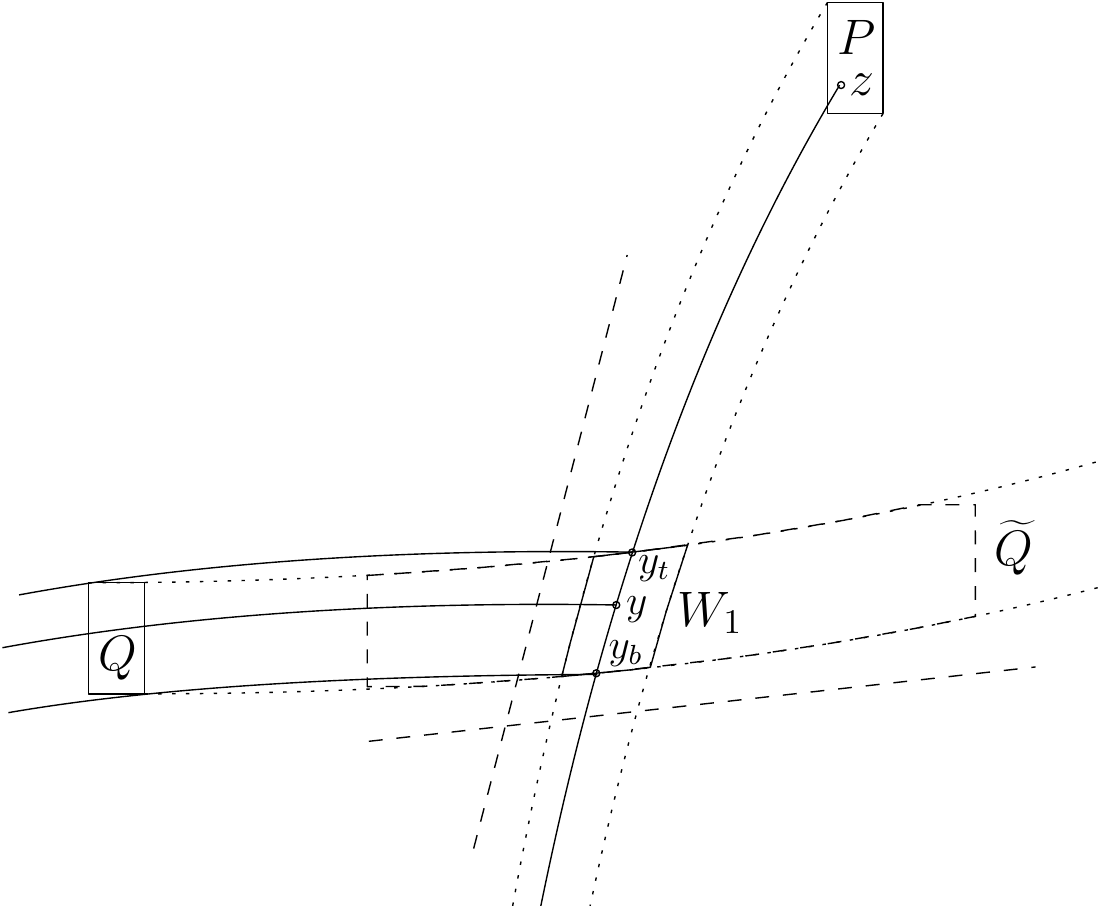}
	\caption{Proof of Key-lemma. Showing as dashed tangents the relative size $A_2>>A_1.$
	} 
	\label{fig:key}
\end{fig}
Lemma \ref{lem:key} is the covering property of the original rectangle $Q$ in terms of parabolic segments and the main component of what in the end allows us to obtain the bound on the line \eqref{A33}.

To setup everything else, we need to specify a few more minor variations of the lemmas that appear in \cite{Oik2022}.
The exact value of the constant $C$ appearing Lemma \ref{lem:geom12} below depends on the fixed parameters $N_1,N_2,A_2/A_1$ appearing on the line \eqref{eq:parameters}, its precise value however is not important.
\begin{lem}\label{lem:geom12}  Let $Q = I\times J \in\calR^{\gamma}$ and $x\in Q.$ Then, there exists an absolute constant $C$ so that   
	\begin{align*}
		\lim_{A_1\to\infty}\frac{\abs{I(x,+,W_1)}}{\ell(I)} = C.
	\end{align*}
	with uniform convergence independent of the data $x,Q.$
\end{lem}
\begin{proof} The details are completely analogous to those in \cite[Lemma 3.3.]{Oik2022}. 
\end{proof}
\begin{lem}\label{lem:geom1}	
	Let $x\in Q$ be arbitrary. Then,
	\begin{align*}
		\abs{I(x,+,W_1)}\sim \ell(I),\qquad \lim_{A_1\to\infty}\sup_{\substack{Q\in\calR^{\gamma} \\ x,x'\in Q }} \frac{\abs{I(x,+,W_1)}}{\abs{I(x',+,W_1)}} = 1.
	\end{align*}
\end{lem}
\begin{proof} Follows immediately from Lemma \ref{lem:geom12}.
\end{proof}

\subsection{Auxiliary functions} In this section we comment on how to choose the auxiliary functions $g_Q,$ $g_{W_1}$ and $g_{P}.$
We first take a look at the rectangle $P$. For this, define
$$
R_{\delta_1,\delta_2}(r) = \delta_1\Big[0,2^{-(r-1)}\frac{\ell(I)}{A_1}\Big]\times \delta_2\big[0, 2^{-r}\ell(I)^2\big],\qquad \delta_1,\delta_2\in\{-1,1\},\qquad 1\leq r <\infty.
$$ 
Let $v_{lt}$ and $v_{rb}$ stand for the left-top and right-bottom vertices of $P$ and define
\begin{align*}
	P^{lt}_r= v_{lt} + R_{1,-1}(r),\qquad P^{rb}_r = v_{rb} +R_{-1,1}(r),
\end{align*}
parts of their boundaries
\begin{equation*}
	\Delta_r(P^{lt}_r) =  \overline{\partial P^{lt}_r\setminus\partial P}, \qquad  
	\Delta_r(P^{rb}_r)  =\overline{\partial P^{rb}_R\setminus\partial P},
\end{equation*}
and the central area of $P,$
\begin{equation*}
	P^c  =  P\setminus\big(P^{lt}_1\cup P^{rb}_1\big).
\end{equation*}
See Figure \ref{fig:blowup}.
If $z\in P^c$ and $y\in\phi(z,+,W_1),$ then $\babs{I(y,+,P)}\sim \frac{\ell(I)}{A_1}.$ In fact, Lemma \ref{lem:geom2} shows that the sets $\Delta_r(P^{lt}_r),	\Delta_r(P^{rb}_r)$ exactly quantify this. The proof of Lemma \ref{lem:geom2} is analogous to the proof of \cite[Lemma 3.1.1.]{Oik2022}, we omit the details.
\begin{lem}\label{lem:geom2} 
	Let $z\in P^c$ and $y\in\phi(z,-,W_1),$ then 
	\begin{align*}
		\abs{I(y,+,P)}\sim \frac{\ell(I)}{A_1}.
	\end{align*}
	Let $z\in \Delta_r(P^{lt}_r)\cup \Delta_r(P^{rb}_r)$ and $y\in\phi(z,-,W_1).$ Then, 
	\begin{align*}
		\abs{I(y,+,P)}&\sim 2^{-r}\frac{\ell(I)}{A_1}.
	\end{align*}
\end{lem}
We still need to specify the auxiliary functions $g_Q,g_{W_1},g_P.$  The choices 
$g_Q = 1_Q,$ and $g_P = 1_P$
are particularly simple, the choice of $g_{W_1}$ is more delicate, however. As we expand (in the first iteration \eqref{A1}, \eqref{A2}, but the same problem is in all iterations) we end up with
\begin{align*}
	h_{W_1} = \frac{g_{W_1} H_{\gamma}h_Q}{H_{\gamma}g_P}.
\end{align*} 
The denominator $H_{\gamma}g_P$ vanishes towards $\partial_lW_1\cup\partial_rW_1,$ see Figure \ref{fig:blowup}, and to attenuate the resulting blow-up we choose $g_{W_1}$ so that 
$\abs{g_{W_1}}\sim 1_{W_1}\abs{H_{\gamma}g_P}.$
\begin{fig}[h]
	\centering
	\includegraphics[scale=0.9]{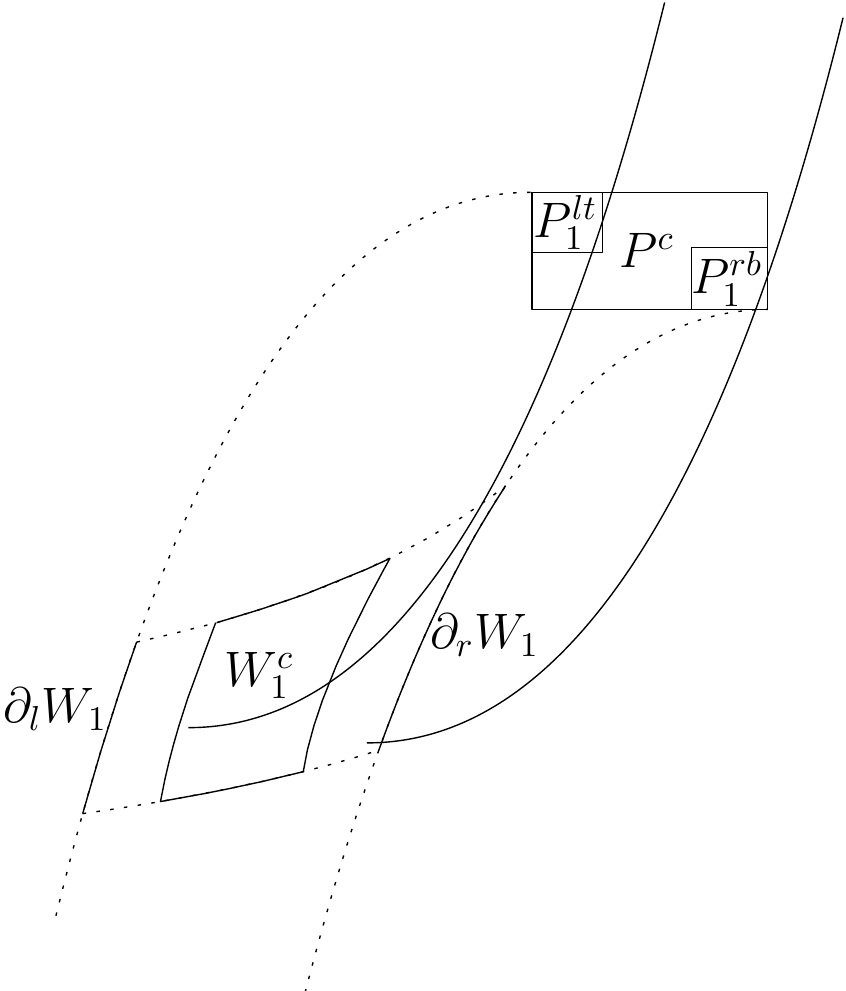}
	\caption{The function $1/H_{\gamma}g_P$ behaves well on $W_1^c$ and blows up towards $\partial_lW_1$ and $\partial_rW_1.$
	} 
	\label{fig:blowup}
\end{fig}
To pull of everything in the end, it helps to construct the functions $\{g_{W_1}\}_{Q\in\calR^{\gamma}}$ in a different way than simply letting the previous "$\sim$" be an "$=$" as then we have the following desired properties ready to go:
\begin{enumerate}[$(i)$]
	\item there holds that
	\begin{align*}
		1_{W_1^c}g_{W_1} = 1_{W_1^c},\qquad W_1^c= \big\{y\in W_1: \exists z\in P^c: y\in \phi(z,+,W_1)\big\},
	\end{align*}
	\item and for all $y\in W_1$ that
	\begin{align*}
		g_{W_1}(y)\sim\abs{I(y,+,P)}\frac{A_1}{\ell(I)},
	\end{align*}
	where the implicit constants do not depend on the data $y,Q,$ 
	\item and
	\begin{align*}
		\lim_{A_1\to\infty}\sup_{\substack{ Q\in\calR^{\gamma} \\ z\in P \\  t,t'\in I(z,+,W_1) }}\frac{g_{W_1}(z+\gamma(t))}{g_{W_1}(z+\gamma(t'))} = 	 1.
	\end{align*}
\end{enumerate}
 We direct the reader to \cite[Section 3.1.2.]{Oik2022}, where the family $\{g_{W_1}\}_{Q\in\calR^{\gamma}}$is defined in the setup of $\gamma_{\op{adj}},$ the details for how to achieve the above properties $(1),(2)$ and $(3)$ are completely analogous and we omit them.
As far as the technical apparatus is concerned we are done with the modifications. Now the claims on the lines \eqref{A3} and \eqref{A33} follow as in \cite[Section 3.2.1.]{Oik2022} by choosing the absolute constant $A_1>>1$ large enough.

 We still need to make sure that we can iterate \eqref{A5} and obtain \eqref{A6} and \eqref{A66}. Given that we have already demonstrated the first step, the argument is now entirely geometric and encapsulated in Figure \ref{fig:geom}. We simply begin from $P$ and backtrace everything back to $Q.$
 \begin{proof}[Proof of Theorem \ref{thm:osc}]
	Recall that we have specified to the parabolic curves $\gamma_{\op{opp}},\gamma_{\op{adj}}.$
	The information \eqref{eq:osc1} follows immediately from how the sets were defined in \eqref{eq:auxintervals}, \eqref{eq:parameters} and \eqref{eq:geometry1}. The bounds \eqref{eq:osc2} follow immediately from those recorded on the lines \eqref{A3} and \eqref{A6}. 
	The left bound of \eqref{eq:osc3} is proved by only \eqref{A33} and \eqref{A66},
	see how the proof of Theorem \ref{thm:adjacent} is concluded in \cite[Section 3.3.]{Oik2022}.
	
	For the second estimate of \eqref{eq:osc3} is new and we provide all details.
	From the definition of the intervals $I_{A_1},I_{A_2},$ there holds that 
	\[
	I(x,-,S_1) \subset \wt{I} := 10\big(I_{A_1}\cup -I_{A_1}\cup I_{A_2}\cup -I_{A_2}\big),\qquad x\in S_2,
	\]
	where the dilation by $10$ is centric, separately for each of the intervals.
	Thus,
	\begin{align*}
		\babs{	\bave{[b,H_{\gamma}]\psi_{S_1},\psi_{S_2}}} &= \Babs{\int_{\R^2}\psi_{S_2}(x)\int_{\R}\left(b(x)-b(x-\gamma(t))\right)\psi_{S_1}(x-\gamma(t))\frac{\ud t}{t}\ud x }\\
		&\qquad\lesssim \int_{S_2} \int_{I(x,-,S_1)}\babs{b(x)-b(x-\gamma(t)}\frac{\ud t}{\abs{t}}\ud x \\ 
		&\qquad\qquad \lesssim \int_{\wt{I}}\int_{S_2} \babs{b(x)-b(x-\gamma(t)}\ud x \frac{\ud t}{\abs{t}}.
	\end{align*}
	There exists an absolute constant $C=C_{\gamma}$ such that for any $t\in \wt{I}$ and $x\in S_2$ there holds that $x-\gamma(t)\in (C\Lambda)^{\beta}Q.$ Thus,
	\begin{align*}
		\int_{S_2} \babs{b(x)-b(x-\gamma(t)}\ud x &\leq \int_{S^2} \babs{b(x)-\bave{b}_{(C\Lambda)^{\beta}Q}}\ud x \\ 
		&\qquad+ \int_{S^2} \babs{\bave{b}_{(C\Lambda)^{\beta} Q}-b(x-\gamma(t)}\ud x \\
		&\qquad\qquad \lesssim \int_{(C\Lambda)^{\beta}Q} \babs{b(x)-\bave{b}_{(C\Lambda)^{\beta}Q}}\ud x.
	\end{align*}
	Thus, 
	\[
	\babs{	\bave{[b,H_{\gamma}]\psi_{S_1},\psi_{S_2}}} \lesssim \int_{\wt{I}} \int_{(C\Lambda)^{\beta}Q} \babs{b(x)-\bave{b}_{(C\Lambda)^{\beta}Q}}\ud x\frac{\ud t}{\abs{t}}\lesssim \int_{(C\Lambda)^{\beta}Q} \babs{b(x)-\bave{b}_{(C\Lambda)^{\beta}Q}}\ud x,
	\]
	where in the last estimate we used that $\ud t/\abs{t}$ is a Haar measure.
\end{proof}

\section{Necessity for $q\geq p$}
In this section we prove Theorems \ref{thm:lb} and \ref{thm:nec}.
\subsection{Boundedness} Theorem \ref{thm:lb} is an instant consequence of Theorem \ref{thm:osc}.
\begin{proof}[Proof of Theorem \ref{thm:lb}] Let $Q\in \calR^{\gamma}$ and $\psi_{S_i}$ be the functions generated by Theorem \ref{thm:osc}. Then, we have 
	\begin{align*}
		\int_Q\abs{b-\ave{b}_Q} &\lesssim \babs{\ave{[b,H_{\gamma}]\psi_{S_1},\psi_{S_2}}} \\ 
		&= \babs{\bave{ \lambda[b,H_{\gamma}]\psi_{S_1} , \lambda^{-1}\psi_{S_2} }} \lesssim \Norm{[b,H_{\gamma}]}{L^p_{\mu}\to L^q_{\lambda}}\Norm{\psi_{S_1}}{L^p_{\mu}}\Norm{\psi_{S_2}}{L^{q'}_{\lambda^{-1}}}.
	\end{align*}
	Recall that $\mu\in A_{p,p}^{\gamma}$ if and only if $\mu^p\in A_p^{\gamma};$ and   that $\lambda\in A_{q,q}^{\gamma}$ if and only if $\lambda^{-1}\in A_{q',q'}^{\gamma}$ if and only if $\lambda^{-q'}\in A_{q'}^{\gamma}.$ Hence both $\mu^p$ and $\lambda^{-q'}$ are doubling; this means that for any $\gamma$-cube $Q\in\calR^{\gamma} = \calR^{\beta}$ there holds that 
	\begin{align}
		\mu^p(2^{\beta}Q)\lesssim \mu^p(Q),\qquad 	\lambda^{-q'}(2^{\beta}Q)\lesssim 	\lambda^{-q'}(Q).
	\end{align}
	Then, by $\abs{\psi_{S_i}}\lesssim 1_{S_i}$ and \eqref{eq:osc1},  we find
	\begin{align*}
		\Norm{\psi_{S_1}}{L^p_{\mu}}\Norm{\psi_{S_2}}{L^{q'}_{\lambda^{-1}}} &\lesssim \mu^p(S_1)^{1/p}\lambda^{-q'}(S_2)^{1/q'} \lesssim \mu^p(Q)^{1/p}\lambda^{-q'}(Q)^{1/q'}.
	\end{align*}
	From $\mu\in A_{p,p}$ and $\lambda \in A_{q,q}$ it follows (see \cite[Proposition 3.1.]{HOS2022}) that 
	\[
	\mu^p(Q)^{1/p}\lambda^{-q'}(Q)^{1/q'}\lesssim_{[\mu]_{A_{p,p}^{\gamma}} [\lambda]_{A_{q,q}^{\gamma}}}\ \nu(Q)^{1/p+1/q'} = \nu(Q)^{1+\alpha/\abs{\beta}},
	\]
	concluding the proof.
\end{proof}

\subsection{Compactness}
We turn to compactness and prove Theorem \ref{thm:nec}. We follow the approach of Uchiyama \cite{Uch1978}, for an alternative presentation see \cite{HOS2022}. Theorem \ref{thm:osc} does not have as much freedom as the corresponding one for non-degenerate SIOs does (implicit in \cite{HyLpLq}, and for a similar layout as here see \cite[Proposition 4.2.]{HOS2022}), hence in the proof we need to more carefully exploit the precise form of the sets $Q,W_1,W_2,P.$
The key difference to the proof in \cite{HOS2022} is the replacement of \cite[Lemma 5.21.]{HOS2022} with the following.
\begin{lem}\label{lem:seq}
	Let $1<p,q<\infty$, let $\mu$ and $\lambda$ be arbitrary weights, and let $U\colon L_\mu^p\to L_\lambda^q$ be a bounded linear operator. Then, there does not exist a sequence $\{u_i\}_{i=1}^\infty$ with the following properties:
	\begin{enumerate}[$(i)$]
		\item $\sup_{i\in\N}\|u_i\|_{L_\mu^p}\lesssim 1,$
		\item for all sequences $\Norm{\{\beta_i\}_i}{\ell^{\infty}}\lesssim 1$ there holds that 
		\[
		\big\|\sum_{i=m+1}^k \beta_iu_i\big\|_{L_\mu^p}
		\lesssim \big(\sum_{i=m+1}^k\|\beta_iu_i\|_{L_\mu^p}^p\big)^\frac{1}{p},\qquad k\geq m,
		\] 
		\item there exists $\Phi\in L^q_{\lambda}$ such that  $$\lim_{i\to\infty} \|\Phi-U(u_i)\|_{L_\lambda^q} = 0,\qquad\|\Phi\|_{L_\lambda^q}>0.$$
	\end{enumerate} 
\end{lem}
The proof is identical to \cite[Lemma 5.21.]{HOS2022}, where the assumption
\begin{align}
\{x\colon u_i(x)\neq 0\}\cap \{x\colon u_j(x)\neq 0\}=\emptyset,\quad i\neq j
\end{align}
guarantees the validity of the point $(ii)$ in the proof.

\begin{proof}[Proof of Theorem \ref{thm:nec}]  As $[b,H_{\gamma}]$ is compact, it is bounded and by Theorem \ref{thm:lb} we know that $b\in\BMO_{\nu}^{\gamma,\alpha}.$ To reach a contradiction, we assume on the contrary that $b\not\in\VMO_{\nu}^{\gamma,\alpha},$ that is $b\in \BMO_{\nu}^{\gamma,\alpha}\setminus \VMO_{\nu}^{\gamma,\alpha}.$ Thus, at least one of the conditions \eqref{eq:VMOv1},  \eqref{eq:VMOv2},  \eqref{eq:VMOv3} fails. If any of them fails, we construct a sequence $\{u_i\}_{i=1}^{\infty}$ of functions as in Lemma \ref{lem:seq}, thus contradict the boundedness of $[b,H_{\gamma}]$ and conclude that $b\in\VMO_{\nu}^{\gamma,\alpha}.$
	
	If any of the three conditions fails, we are guaranteed a sequence of $\gamma$-cubes $\{Q^j\}_{j=1}^{\infty}\subset\calR^{\gamma}$ such that $1\lesssim \calO_{\nu}^{\gamma,\alpha}(b;Q^j).$ By Theorem \ref{thm:osc} it follows that
	\begin{align}\label{st1}
		1\lesssim \nu^{-(1+\alpha/\abs{\beta})}(Q^j)\babs{\bave{[b,H_{\gamma}]\psi_{S_1^j},\psi_{S_2^j}}},
	\end{align}
	where $\psi_{S_1^j},\psi_{S_2^j}$ are as on the line \eqref{eq:osc1} and \eqref{eq:osc2}. Using 
	$$
	\nu^{1+\alpha/\abs{\beta}}(Q^j)\sim \mu^{p}(Q^j)^{1/p}\lambda^{-q'}(Q^j)^{1/q'},
	$$ see \cite[Proposition 3.1.]{HOS2022}, the bound \eqref{st1} is equivalent with the first bound in the following estimate \eqref{st2}, the second is H\"{o}lder's inequality, and for the third we use the doubling property of $\lambda^{-q'}$ coupled with the information \eqref{eq:osc1}:
	\begin{equation}\label{st2}
		\begin{split}
			1&\lesssim   \Babs{\Bave{\lambda[b,H_{\gamma}]\Big(\frac{\psi_{S_1^j}}{\mu^{p}(Q^j)^{1/p}}\Big),\lambda^{-1}\frac{\psi_{S_2^j}}{ \lambda^{-q'}(Q^j)^{1/q'}}}} \\ 
			&\lesssim \Norm{[b,H_{\gamma}]u_j}{ L^q_{\lambda}}\BNorm{\frac{\psi_{S_2^j}}{ \lambda^{-q'}(Q^j)^{1/q'}}}{L^{q'}_{\lambda^{-1}}} \lesssim \Norm{[b,H_{\gamma}]u_j}{ L^q_{\lambda}},\qquad u_j := \psi_{S_1^j}/\mu^{p}(Q^j)^{1/p}.
		\end{split}
	\end{equation}
	By the doubling property of $\mu^{p}$ we obtain $\Norm{u_j}{L^p_{\mu}}\lesssim 1.$ It follows from compactness that there exists a subsequence of $\{[b,H_{\gamma}]u_j\}_{j=1}^{\infty}$ with a limit $\Phi\in L^{q}_{\lambda},$ and by \eqref{st2} $\|\Phi\|_{L_\lambda^q}>0,$ thus the property $(iii)$ of Lemma \ref{lem:seq} is satisfied. We conclude that whenever the $\VMO_{\nu}^{\gamma,\alpha}$ condition fails we can construct a sequence of functions satisfying the properties $(i)$ and $(iii)$ of Lemma \ref{lem:seq}. It remains to show that in all three cases we can also arrange the property $(ii).$
	
	If the condition \eqref{eq:VMOv3} fails, we obtain a sequence satisfying
	\begin{align}\label{eq:case1}
		\lim_{j\to\infty}\dist(Q^j,0) = \infty.
	\end{align}  
	From \eqref{eq:case1} we find $\gamma$-cubes ever farther away from the origin and extract a subsequence such that $S_1^{k}\cap S_1^{l} = \emptyset$ if $k\not=l,$ where we understand that $S^j_1$ is the set generated by Theorem \ref{thm:osc} for a fixed $Q^j.$ Thus by the support localization \eqref{eq:osc2} we achieve
	$$ 
	\{x\colon u_k(x)\neq 0\}\cap \{x\colon u_l(x)\neq 0\} = \emptyset,
	$$ which implies $(ii),$ by the proof of Lemma \ref{lem:seq}.
	
	If \eqref{eq:VMOv2} fails, then we obtain a sequence 
	\begin{align}\label{eq:case2}
		\lim_{j\to\infty}\ell(Q^j)= \infty.
	\end{align}
	Note that for any finite sequence of rectangles $\{Q^j\}_{j=1}^N\subset \calR^{\gamma},$ if $\ell(Q^{N+1})$ is sufficiently large, then at least one of
	$S_1^{N+1}, S_2^{N+1}$ does not intersect $\cup_{j=1}^NQ^j;$ indeed, by \eqref{eq:osc1} we have $\dist(S_1^k,S_2^k) \to \infty$ as $k\to \infty$, thus it is enough to pick the $(N+1)$th element in the subsequence to satisfy $\dist(S_1^{N+1},S_2^{N+1}) > \diam(\cup_{j=1}^NQ^j).$ 
	It follows that we are guaranteed an indexing sequence $\{Q^j\}_{j=1}^{\infty}$ such that $\{S_{i(j)}^j\}_{j=1}^{\infty}$ is pairwise disjoint with some choice of $i(j)\in\{1,2\}.$ Let 
	$\{S_1^j\}_{j=1}^{\infty}$ and $\{S_2^j\}_{j=1}^{\infty}$
	be the sub-sequences where $i=1$ and $i=2,$ respectively. Clearly both are pairwise disjoint and at least one has infinite length. If $\{S_1^j\}_{j=1}^{\infty}$ has infinite length, then we are done. If $\{S_2^j\}_{j=1}^{\infty}$ has infinite length, then we use that $[b,H_{\gamma}]$ is bounded if and only if $[b,H_{\gamma}^*]$ is bounded, and
	\[
	\bave{[b,H_{\gamma}]\psi_{S_1^j},\psi_{S_2^j}} = -\bave{[b,H_{\gamma}^*]\psi_{S_2^j},\psi_{S_1^j}}.
	\] 
	Then, the argument proceeds as above.
	
	Lastly, we suppose that the condition \eqref{eq:VMOv1} fails and 
	\begin{align}\label{eq:st2}
		\lim_{j\to\infty}\ell(Q^j) = 0.
	\end{align}
	We find a subsequence $\{y_{j}\}_{j=1}^{\infty}\subset \{u_j\}_{j=1}^{\infty}$ and a constant $C>0$
such that
	\begin{align}\label{eq:ii}
		\big\|\sum_{j\in J}\beta_jy_j\big\|_{L_\mu^p}
		\leq C \bigg(\sum_{j\in J}\|\beta_jy_j\|_{L_\mu^p}^p\bigg)^\frac{1}{p},
	\end{align}
	for all $\abs{\beta_j}\leq 1$ and $J\subset \N,$ thus checking the point $(ii)$ of Lemma \ref{lem:seq}. 
	When $J = \{1\}$ and $y_1 = u_1$, the bound \eqref{eq:ii} holds with $C=1.$ 
	Suppose that we have a subsequence $\{y_j\}_{j=1}^N = \{u_{i(j)}\}_{j=1}^N,$ $i(j+1)>i(j),$ such that \eqref{eq:ii} holds for all $J\subset [1,2,\dots,N]$ and with a constant $C<2.$ Then, we show how to pick $y_{N+1}\in \{u_j\}_{j=i(N)+1}^{\infty}$ from the tail so that \eqref{eq:ii} is satisfied for all  $J\subset [1,2,\dots,N+1]$ and with another constant $C'\in [C,2).$ Then by recursion \eqref{eq:ii} holds with $C=2$ and we are done.

	If there exists $u_k\in \{u_j\}_{j=i(N)+1}^{\infty}$ such that 
	\begin{align}\label{eq:st3}
		S_1^{k}\cap  \big(\cup_{j=1}^NS_1^{i(j)}\big)  = \emptyset,
	\end{align}
	we let $y_{N+1} = u_k.$ Then by $\eqref{eq:st3}$ and the inductive hypothesis,
	\begin{align*}
		\big\|\sum_{j\in J}\beta_jy_j\big\|_{L_\mu^p} &= \Big(1_{J}(N+1)\Norm{\beta_{N+1}y_{N+1}}{L^p_{\mu}}^p + 	\big\|\sum_{j\in J\setminus\{N+1\}}\beta_jy_j\big\|_{L_\mu^p}^p\Big)^{\frac{1}{p}}  \\ 
		&\leq \Big(1_{J}(N+1)\Norm{\beta_{N+1}y_{N+1}}{L^p_{\mu}}^p + C \sum_{j\in J\setminus\{N+1\}}\hspace{-1em}\big\|\beta_jy_j\big\|_{L_\mu^p}^p\Big)^{\frac{1}{p}} \\
		&\leq \max(1,C)\bigg(\sum_{j\in J}\|\beta_jy_j\|_{L_\mu^p}^p\bigg)^\frac{1}{p},
	\end{align*}
	and as $C' := \max(1,C)<2$ the choice $y_{N+1}$ is as desired.
	
	If \eqref{eq:st3} is not satisfied, then $S_1^k\subset B_{\beta}(0,R),$ for a sufficiently large $R$ and all $k > i(N).$
	We denote $y_{N+1} = u_{k}$ for some $u_k\in \{u_j\}_{j=i(N)+1}^{\infty}.$
	By $\supp(y_{N+1})\subset S_1^{k},$ we write
	\begin{align}\label{eq:st4}
		\bNorm{\sum_{j\in J} \beta_jy_j }{L^p_{\mu}} =  \Big(\bNorm{1_{S_1^{k}}\sum_{j\in J} \beta_jy_j }{L^p_{\mu}}^p + \bNorm{1_{\R^2\setminus S_1^{k}}\sum_{j\in J\setminus\{N+1\}} \beta_jy_j }{L^p_{\mu}}^p\Big)^{\frac{1}{p}}.
	\end{align}
	By the inductive hypothesis the latter term is bounded as desired,
	\begin{align}\label{eq:st5}
		\bNorm{1_{\R^2\setminus S_1^{k}}\sum_{j\in J\setminus\{N+1\}} \beta_jy_j }{L^p_{\mu}}^p\leq C\sum_{j\in J\setminus\{N+1\}}\bNorm{\beta_jy_j } {L^p_{\mu}}^p.
	\end{align}
	For the left term, we estimate
	\begin{align*}
		\bNorm{1_{S_1^{k}}\sum_{j\in J} \beta_jy_j }{L^p_{\mu}}\leq  	\bNorm{1_{J}(N+1)1_{S_1^{k}}\beta_{N+1}y_{N+1} }{L^p_{\mu}} + \bNorm{1_{S_1^{k}}\psi_N}{L^p_{\mu}},
	\end{align*}
    where $ \psi_N = \sum_{j\in J\setminus\{ N+1\}} \beta_jy_j.$
	By that	$\psi_N$ is a bounded function, that $S_1^k\subset B_{\beta}(0,R),$ and the $A_{\infty}^{\gamma}$-property of $\mu^p,$ there exists some $\delta=\delta_{\mu}>0$ so that 
	$$
	\Norm{1_{S_1^{k}}\psi_N}{L^p_{\mu}}\lesssim_{\psi_N} \mu^{p}(S_1^k)\lesssim_{\mu} \Big(\frac{\abs{S_1^k}}{\abs{B_{\beta}(0,R)}}\Big)^{\delta}\mu^{p}(B_{\beta}(0,R))\lesssim_{\mu,R,\delta} \abs{S_1^k}^{\delta}.
	$$
	Let $\theta>0$ be arbitrary and let $k>i(N)$ be so large that 
	$$
	\Norm{1_{S_1^{k}}\psi_N}{L^p_{\mu}}\leq	\theta\sum_{j\in J}\bNorm{\beta_jy_j } {L^p_{\mu}}^p.
	$$
	Continuing from \eqref{eq:st4} we find
	\begin{align*}
			\big\|\sum_{j\in J}\beta_jy_j\big\|_{L_\mu^p} &\leq \theta\sum_{j\in J}\bNorm{\beta_jy_j } {L^p_{\mu}}^p + C\sum_{j\in J\setminus\{N+1\}}\bNorm{\beta_jy_j } {L^p_{\mu}}^p \leq (C+\theta)\sum_{j\in J}\bNorm{\beta_jy_j } {L^p_{\mu}}^p
	\end{align*}
	so it remains to demand that $C':= C+\theta<2.$
\end{proof}

\section{The case $q< p$}\label{sect:q<p} 
\begin{proof}[Proof of Theorem \ref{thm:main:q<pA}]
	We show that $[b,H_{\gamma}]$ can be approximated by a compact operator.
	From the boundedness of $H_{\gamma},$ see e.g. Stein, Wainger \cite{SteWai1978} and Duoandikoetxea, Rubio de Francia \cite{DuoRdF1986}, and H\"{o}lder's inequality, we obtain 
	$$
	\Norm{[b,H_{\gamma}]}{L^p\to L^q} \leq \Norm{bH_{\gamma}}{L^p\to L^q} + \Norm{H_{\gamma}b}{L^p\to L^q}\lesssim \Norm{b}{L^r}.
	$$
	The commutator being unchanged modulo additive constants in the symbol $b,$ we have
	$$
		\Norm{[b-\wt{b},H_{\gamma}]}{L^p\to L^q} = \inf_{c\in\C}\Norm{[b-\wt{b}-c,H_{\gamma}]}{L^p\to L^q} \lesssim \inf_{c\in\C}\Norm{b-\wt{b}-c}{L^r} = \Norm{b-\wt{b}}{\dot L^r}.
	$$
	Thus, from the density of $C^{\infty}_c$ in $\dot L^r,$ we may assume that $b\in C^{\infty}_c.$ Let $\supp(b)\subset B_{\beta}(0,M).$ 
	
	With the auxiliary function $\chi^{0,R}= \chi_0^{0,R},$ recall \eqref{eq:chi}, we split
	\begin{align}\label{eq:tred0}
		\begin{split}
				H_{\gamma} =\left(\chi^{0,R}H_{\gamma}\chi^{0,R}+\chi^{0,R}H_{\gamma}\chi^{R,\infty}+\chi^{R,\infty}H_{\gamma}\chi^{0,R}\right)+\chi^{R,\infty}H_{\gamma}\chi^{R,\infty}.
		\end{split}
	\end{align}
Taking $R>M$ implies $\supp(b\chi^{R,\infty}) = \emptyset,$ and thus
\begin{equation}
	\begin{split}\label{eq:tred1}
		[b,\chi^{R,\infty}H_{\gamma}\chi^{R,\infty}] = b\chi^{R,\infty}H_{\gamma}\chi^{R,\infty}- \chi^{R,\infty}H_{\gamma}b\chi^{R,\infty} = 0.
	\end{split}
\end{equation}
By \eqref{eq:tred0} and \eqref{eq:tred1} it remains to show that 
\begin{align}\label{eq:tred2}
		[b,Y_{\gamma}]\in\calK(L^p,L^q),\qquad Y_{\gamma}\in \left\{ \chi^{0,R}H_{\gamma}\chi^{0,R}, \chi^{0,R}H_{\gamma}\chi^{R,\infty}, \chi^{R,\infty}H_{\gamma}\chi^{0,R}\right\}.
\end{align}
We only show \eqref{eq:tred2} for the choice $Y_{\gamma} = \chi^{0,R}H_{\gamma}\chi^{0,R},$ the rest are completely analogous.
We have $C^{\infty}_c\subset \VMO^{\beta,0},$ by \eqref{eq:CMObeta=VMObeta}, and thus $[b,H_{\gamma}] \in \calK(L^p,L^p),$ by Theorem \ref{thm:suff}.
Moreover, by H\"{o}lder's inequality the compactly supported multiplier $\chi^{0,R}$ is $L^p$-to-$L^q$ bounded, for all $q\leq p.$ As the composite of a compact operator with bounded operators is compact, we obtain 
$$
[b,Y_{\gamma}]= \chi^{0,R}\circ [b,H_{\gamma}] \circ \chi^{0,R}\in\calK(L^p,L^q).
$$
\end{proof}

\begin{thm}[\cite{HLTY2022}] Let $1<q<p<\infty,$ let $1/q = 1/r+1/p$ and $b\in \dot L^r(\R^n),$ let $T$ be a Calderón-Zygmund operator.  Then, there holds that $[b,T]\in\calK(L^p(\R^n),L^q(\R^n)).$ 
\end{thm}
\begin{proof} 
	We show that $[b,T]$ can be approximated by a compact operator.
	From the boundedness of $T$ and H\"{o}lder's inequality we have
	$$
	\Norm{[b-\wt{b},T]}{L^p\to L^q} \lesssim \Norm{b-\wt{b}}{\dot L^r}.
	$$
	Thus from the density of $C^{\infty}_c$ in $\dot L^r,$ we may assume that $b\in C^{\infty}_c.$ We split
	\begin{align}\label{eq:tred00}
	\begin{split}
		T =\left(\chi^{0,R}T\chi^{0,R}+\chi^{0,R}T\chi^{R,\infty}+\chi^{R,\infty}T\chi^{0,R}\right)+\chi^{R,\infty}T\chi^{R,\infty},
	\end{split}
\end{align}
where $\chi^{a,b} := 1_{B(0,b)}\setminus 1_{B(0,a)}.$
Taking $R>M$ implies $\supp(b\chi^{R,\infty}) = \emptyset$ and thus
\begin{equation}
	\begin{split}\label{eq:tred11}
		[b,\chi^{R,\infty}T\chi^{R,\infty}] =  b\chi^{R,\infty}T\chi^{R,\infty}- \chi^{R,\infty}Tb\chi^{R,\infty} = 0.
	\end{split}
\end{equation}
By \eqref{eq:tred00} and \eqref{eq:tred11} it remains to show that 
\begin{align}\label{eq:tred22}
	[b,Y]\in\calK(L^p,L^q),\qquad Y\in \left\{ \chi^{0,R}T\chi^{0,R}, \chi^{0,R}T\chi^{R,\infty}, \chi^{R,\infty}T\chi^{0,R}\right\}.
\end{align}
We only show \eqref{eq:tred22} for the choice $Y = \chi^{0,R}T\chi^{R,\infty},$ the rest are completely analogous.
Since $C^{\infty}_c\subset \VMO$ it follows that $[b,T] \in \calK(L^p,L^p),$ by Uchiyama \cite{Uch1978}.
Moreover, the multiplier $\chi^{R,\infty}$ is $L^p$-to-$L^p$ bounded, and the multiplier $\chi^{0,R}$ is $L^p$-to-$L^q$ bounded, for all $q\leq p.$  As the composite of a compact operator with bounded operators is compact, we obtain $$[b,Y] =  \chi^{0,R}\circ [b,T] \circ \chi^{R,\infty}\in\calK(L^p,L^q).$$
\end{proof}

\section*{Acknowledgements} I thank the research group of Katrin F\"{a}ssler and Tuomas Orponen and the department of mathematics of the University of Jyv\"{a}skyl\"{a} for their hospitability during the past academic year. I thank Carlos Mudarra for informative discussions relating to Section \ref{sect:density}.

The author was supported by the Academy of Finland through the Project 343530. 
\bibliography{references}
\end{document}